\documentclass{amsart}
\usepackage{amsmath}
\usepackage{amssymb}
\usepackage{amsfonts}
\usepackage{amscd}
\usepackage{mathrsfs}
\usepackage{amsthm}

\newcommand{\CC}{{\rm\bf C}}
\newcommand{\RR}{{\rm\bf R}}
\newcommand{\QQ}{{\rm\bf Q}}
\newcommand{\ZZ}{{\rm\bf Z}}

\newcommand{\FF}{{\rm\bf F}}

\newcommand{\PP}{{\rm\bf P}}
\newcommand{\Adeles}{{\rm\bf A}}

\newcommand{\OO}{\mathcal {O}}

\DeclareMathOperator{\absNorm}{\mathfrak{N}}

\DeclareMathOperator{\Spec}{\mathrm{Spec}}
\DeclareMathOperator{\diag}{\mathrm{diag}}
\DeclareMathOperator{\Gm}{\mathrm {{\bf G}_m}}
\DeclareMathOperator{\GL}{\mathrm{GL}}
\DeclareMathOperator{\SL}{\mathrm {SL}}
\DeclareMathOperator{\PGL}{\mathrm {PGL}}

\DeclareMathOperator{\Oo}{\mathrm {O}}
\DeclareMathOperator{\U}{\mathrm {U}}
\DeclareMathOperator{\SO}{\mathrm {SO}}

\DeclareMathOperator{\Hom}{\mathrm {Hom}}
\DeclareMathOperator{\kernel}{\mathrm {ker}}

\DeclareMathOperator{\rang}{\mathrm {rank}}
\DeclareMathOperator{\tr}{\mathrm{tr}}

\DeclareMathOperator{\res}{\mathrm{Res}}
\DeclareMathOperator{\Ad}{\mathrm{Ad}}
\DeclareMathOperator{\ad}{\mathrm{ad}}
\DeclareMathOperator{\der}{\mathrm{der}}

\DeclareMathOperator{\Lie}{\mathrm{Lie}}
\newcommand{\zentrum}{\mathscr{C}}

\newcommand{\lieg}{{\mathfrak {g}}}
\newcommand{\lieh}{{\mathfrak {h}}}
\newcommand{\liek}{{\mathfrak {k}}}
\newcommand{\liel}{{\mathfrak {l}}}

\newcommand{\lieo}{{\mathfrak {o}}}
\newcommand{\liep}{{\mathfrak {p}}}
\newcommand{\lieq}{{\mathfrak {q}}}

\newcommand{\liegl}{{\mathfrak {gl}}}
\newcommand{\liesl}{{\mathfrak {sl}}}

\newcommand{\characteristic}{\mathrm {char\,}}
\DeclareMathOperator{\vol}{\mathrm {vol}}

\DeclareMathOperator*{\Otimes}{\ensuremath{\otimes}}

\newcommand{\absnorm}[1]{\ensuremath{\left|\!\left|{#1}\right|\!\right|}}

\newcommand{\mmod}{\mathrm {mod\,}}

\newcommand{\adots}{\ensuremath{
\put(-3,-6){\normalsize$\,$}
\put(-2.2,-3.5){\normalsize$\cdot$}
\cdot
\put(-1.0,3.5){\normalsize$\cdot$}
\put(0.0,6){\normalsize$\,$}
}}

\theoremstyle{plain}
\newtheorem{theorem}{Theorem}[section]
\newtheorem{lemma}[theorem]{Lemma}
\newtheorem{corollary}[theorem]{Corollary}
\newtheorem{proposition}[theorem]{Proposition}

\theoremstyle{remark}
\newtheorem{remark}[theorem]{Remark}
\newtheorem{definition}[theorem]{Definition}

\begin{document}

\title{Modular symbols for reductive groups and $p$-adic Rankin-Selberg convolutions over number fields}
\author{Fabian Januszewski}

\thanks{this research was conducted while the author was visiting Universit\'e Paris 13, France and was supported by the Deutscher Akademischer Austauschdienst.}

\begin{abstract}
We give a construction of a wide class of modular symbols attached to reductive groups. As an application we construct a $\mathfrak{p}$-adic distribution interpolating the special values of the twisted Rankin-Selberg $L$-function attached to cuspidal automorphic representations $\pi$ and $\sigma$ of $\GL_n$ and $\GL_{n-1}$ over a number field $k$. If $\pi$ and $\sigma$ are ordinary at $\mathfrak{p}$, our distribution is bounded and yields analyticity of the associated $\mathfrak{p}$-adic $L$-function.
\end{abstract}

\maketitle

\tableofcontents

\markboth{Fabian Januszewski}{Modular symbols for reductive groups and $p$-adic Rankin-Selberg convolutions}
\section*{Introduction}\label{sec:introduction}

The study of special values of $L$-functions has a long history, dating back to Kummer and Euler. Iwasawa theory provides us with a strong motivation for the construction of $p$-adic $L$-functions. Unfortunately, we still don't know how to construct a $p$-adic $L$-function for a general motive. As a result all known constructions of $p$-adic $L$-functions rely on the study of special values of complex $L$-functions and are essentially automorphic. By Langland's philosophy this seems to be no severe restriction \cite[Question 4.16 et Th\'eor\`eme 5.1]{clozel1990}.

In this paper we study the problem of $\mathfrak{p}$-adic interpolation of the special values of twisted Rankin-Selberg $L$-functions of two automorphic representations $\pi$ and $\sigma$ on $\GL_n$ and $\GL_{n-1}$ (in the sense of \cite{jpss1983,cogdellpiatetskishapiro2004}). Our approach is greatly inspired by \cite{birch1971,manin1972,mazur1972,mazurswinnertondyer1974}, and especially \cite{schmidt1993,kazhdanmazurschmidt2000,schmidt2001}, where Kazhdan, Mazur and Schmidt treat the case $k=\QQ$.

To be more precise, let $\pi$ and $\sigma$ be irreducible cuspidal automorphic representations of $\GL_n(\Adeles_k)$ and $\GL_{n-1}(\Adeles_k)$, where $\Adeles_k$ denotes the ad\`ele ring of a number field $k$. Assuming that $\pi$ and $\sigma$ occur in cohomology and are ordinary at $\mathfrak{p}$, we show the existence of a $\mathfrak{p}$-adic measure $\mu$ which is characterized by the property that for a certain \lq\lq{}period\rq\rq{} $\Omega\in\CC$ we have
$$
\int\chi d\mu=\Omega\cdot L(\frac{1}{2},(\pi\otimes\chi)\times\sigma)
$$
for any Hecke character $\chi$ of finite order with (non-trivial) $\mathfrak{p}$-power conductor (cf. Theorems \ref{satz:distribution}, \ref{satz:distribution2} and \ref{thm:algebraicity} below).

Without the assumption of ordinarity, only assuming that $\pi$ and $\sigma$ are unramified at $\mathfrak{p}$, we still get a $\mathfrak{p}$-adic distribution $\mu$, whose order may be easily bounded by the $\mathfrak{p}$-valuations of the roots of the corresponding Hecke polynomials of $\pi$ and $\sigma$ at $\mathfrak{p}$. In the spirit of Iwasawa theory $\mu$ corresponds to a measure (resp. a distribution) on the Galois group of the maximal abelian extension $k^{\mathfrak{p}^\infty}/k$, unramified outside $\mathfrak{p}$.

A general problem in the field is that we don't know for general $n$ if
$$
\Omega\neq 0.
$$
Of course this assertion is vital for our theory. There are well known affirmative answers in the cases $n=2$ and $n=3$ for totally real $k$, and recent work of Schmidt and Kasten \cite{kastenschmidt2009} extend the case $n=3$ to non-constant coefficients. We don't attack this problem here, since the techniques involved are of a different nature than the problems we discuss. Furthermore we prefer to confine us to a clear treatment involving only cohomology with constant coefficients, since the methods of loc. cit. easily carry over to our setting, but come for the price of additional notation.

In those cases where $\pi$ and $\sigma$ arise from Hilbert modular forms, Shimura's algebraicity result \cite{shimura1978} may be used to relate our periods to Shimura's periods, which may in certain cases eventually be related to Deligne's periods \cite{deligne1979}, notably if $\pi$ and $\sigma$ arise from a base change in $k/\QQ$, \cite{shimura1976,shimura1977,yoshida1994}. In general it remains an open problem to relate $\Omega$ to Deligne's periods if $\pi$ and $\sigma$ arise from motives.

In our attempt to generalize the results of Kazhdan, Mazur and Schmidt (loc. cit) to arbitrary number fields $k$ we encountered two major problems\footnote{We also encountered a minor problem related to the class number of $k$, which had hitherto been overlooked, i.e. in \cite{ashginzburg1994}, cf. the discussion following Theorem \ref{satz:distribution2} below.}.

First, even over $\QQ$, their results on $p$-adic interpolation are incomplete, as there is a technical restriction on the conductors of the twisting characters, depending on $n$. This concerns the so-called {\em generalized local Birch lemmas} of loc. cit., which only hold under the hypothesis that for the Dirichlet characters $\chi$ under consideration $\chi, \chi^2, \dots, \chi^{n-1}$ share the same conductor $\mathfrak{f}\neq 1$. The problem had been attacked before in \cite{utz2004}, where the cases $n=2,3,4$ were settled by computation. We solve the general problem with a new local Birch lemma (Theorem \ref{thm:localbirchlemma} below). As a consequence we can show the unconditional existence of $\mu$, completing also the previously known results over $\QQ$. Our result also suggests a natural generalization for Rankin-Selberg convolutions on $\GL_m\times\GL_n$ for general pairs $m,n$. However, even if we believe that it might be possible to attack the $p$-adic interpolation problem for pairs $(m,n)\neq(n-1,n)$ by similar methods, we content us to give a general global formula in Theorem \ref{thm:globalbirchlemma}.

Second, Kazhdan, Mazur and Schmidt's argument for proving algebraicity and boundedness of the distribution relies on relative modular symbols, as introduced in \cite{schmidt1993}. It turns out that this specific approach does not generalize to the more general setting, especially if $k$ is no more totally real. The reason being that the modular symbols previously considered are defined by means of {\em analytic} Lie groups. To overcome this problem we give an {\em algebraic} interpretation of these modular symbols which enables us to give an abstract treatment of the problem. We are naturally led to consider relative modular symbols attached to morphisms of reductive algebraic groups. 

More precisely let $s:H\to G$ be a $\QQ$-morphism of connected reductive groups over $\QQ$ with $\RR$-anisotropic kernel defined over $\RR$. Denote by $\mathscr X$ and $\mathscr Y^1$ symmetric spaces associated to $G(\RR)$ and $H^{\ad}(\RR)$ respectively, and let $\Gamma$ and $\Gamma^1$ be arithmetic subgroups of the latter groups. Using reduction theory for arithmetic groups, we construct for any subring $A\subseteq \CC$ a pairing
$$
\mathscr P_{s,g}^{q}:H_{\rm c}^{q}(\Gamma\backslash\mathscr X,A)\times H_{\rm c}^{\dim\mathscr Y-q}(\Gamma^1\backslash\mathscr Y^1,A)\to A,
$$
which is functorial in $A$. Specializing to $G=\res_{k/\QQ}\GL_n$, $H=\res_{k/\QQ}\GL_{n-1}$, and $s=\res_{k/\QQ}(\ad\circ j)$ for $\ad\circ j:\GL_{n-1}\to\PGL_n$, we find our special values in the image of this topological symbol, yielding algebraicity and boundedness of the (a priori $\CC$-valued) distribution $\mu$.

Actually many more topological modular symbols may be constructed with our technique. We expect that our method may serve to approach other cases as well.

To give an explicit application of our methods, we deduce the existence of the $\mathfrak{p}$-adic symmetric cube $L$-function of a modular elliptic curve $E$ over a totally real number field $k$. Assuming Langland's functoriality conjecture, we get the existence of all $\mathfrak{p}$-adic odd symmetric power $L$-functions along the same lines.\ \\

\paragraph{{\bf Acknowledgements.}}
The author thanks Claus-G\"unther Schmidt for pointing out some inaccuracies in an earlier draft of this paper. The author also thanks Jacques Tilouine and the Institut Galil\'ee at Universit\'e Paris 13, Paris, France for their friendly hospitality.

\section*{{\rm\em Notation}}

For a set $M$ we write $\#M$ for its cardinality. If a family $(M_i)_{i\in I}$ of sets is given, let
$$
\bigsqcup_{i\in I}M_i
$$
denote the disjoint union of the $M_i$, $i\in I$. For a ring extension $R/S$ and an $S$-module $M$ we let $M_R:=M\otimes_S R$. $\Adeles_k$ (resp. $\Adeles_k^{\rm f}$) denotes the (finite) ad\`ele ring of a global field $k$, $M_k$ denotes the set of places of $k$. If $S$ is a scheme over a ring $R$ and if $A$ is any commutative $R$-algebra, we denote by $S(A)=\Hom_{\Spec R}(\Spec A,S)$ the set of $A$-points of $S$.

Let $G$ be a topological group. Then $G^0$ denotes the connected component of the unit $1\in G$, the same notation applies to algebraic groups with respect to the Zariski topology.

For a finite seperable field extension $k/l$, $\res_{k/l} G$ denotes the restriction of scalars (\`a la Weil) of $k$ to $l$ of a (linear) algebraic group $G$. This is a (linear) algebraic group over $l$, such that for any commutative $l$-algebra $A$ we have $(\res_{k/l}G)(A)=G(A\otimes_l k)$. An isogeny of algebraic groups is an epimorphism with finite kernel. Denote by $G^{\der}$ the commutator group of a linear algebraic group $G$ and by $G^{\ad}$ the adjoint group respectively. So $G^{\ad}$ is the image of $G$ under the adjoint representation $\Ad:G\to\Lie(G)$. Here and in the sequel $\Lie(G)$ is the Lie algebra of $G$. The differential of a morphism $f:G\to H$ of linear algebraic (or of Lie) groups is denoted by $L(f)$. $\mathscr R(G)$ (resp. $\mathscr R_{\rm u}(G)$) is the (unipotent) radical of $G^0$. The $k$-rank of a reductive group over $k$ is the dimension of a maximal $k$-split torus in $G^0$. Note that all maximal $k$-tori are conjugate over $k$ (cf. \cite[Th\'eor\`eme 4.21]{boreltits1965}) and any $k$-torus $T$ in $G^0$ is contained in a maximal torus of $G^0$ defined over $k$ (cf. \cite[2.15 d)]{boreltits1965}).

A morphism $f:G\to H$ of algebraic groups is central if the commutator map $[\cdot,\cdot]:G\times G\to G$ factors as a morphism (and as a map) over $f(G)\times f(G)$. Note that kernels of morphisms defined over $k$ need not be defined over $k$. If a $k$-morphism $f$ is (quasi-)central, then $\kernel f$ is always defined over $k$ \cite[Corollaire 2.12]{boreltits1972}.

\section{Modular symbols for reductive groups}

Throughout this section $G$ is a connected linear algebraic group over $\QQ$. We denote by $X_k(G)$ the $\ZZ$-module of characters $\alpha:G\to\Gm$ defined over a field $k/\QQ$. We define the group
$$
{}^0G:=\bigcap_{\alpha\in X_\QQ(G)}\kernel\alpha^2.
$$
The square on the right hand side is motivated by the fact that $\RR^\times$ has two connected (topological) components. The group ${}^0G$ is normal in $G$ and defined over $\QQ$. The restriction of any $\alpha\in X_\QQ(G)$ on ${}^0G$ is at most of order $2$, and consequently trivial on $\left({}^0G\right)^0$. Therefore we have
\begin{equation}
\left({}^0G\right)^0=\left(\bigcap_{\alpha\in X_\QQ(G)}\kernel\alpha\right)^0.
\label{eq:0G0}
\end{equation}
Let $L$ be a Levi subgroup of $G$ over $\QQ$. Then
$$
G=\mathscr R_{\rm u}G\rtimes L,
$$
which implies
$$
{}^0G=\mathscr R_{\rm u}G\rtimes{}^0L,
$$
since every character on $G$ vanishes on the unipotent radical of $G$. Our interest in ${}^0G$ is motivated by
\begin{proposition}[{\cite[Proposition 1.2]{borelserre1973}}]\label{prop:boreleserre}
Denote by $S$ a maximal $\QQ$-split torus in the radical of $G$. Then
$$
G(\RR)={}^0G(\RR)\rtimes S(\RR)^0.
$$
The group ${}^0G(\RR)$ contains every compact subgroup as well as every arithmetic subgroup of $G(\RR)$.
\end{proposition}

Furthermore we have

\begin{proposition}\label{prop:0g0}
Let $G$ be reductive. Then ${}^0G$ is reductive and we have
$$
X_\QQ\left({}^0G\right)\otimes_\ZZ\QQ=0.
$$
\end{proposition}

\begin{proof}
The group $\left({}^0G\right)^0$ is generated by the semi-simple group $G^{\der}$ and a central torus, which implies that ${}^0G$ is reductive.

To prove the second assertion, we note that the canonical inclusion $\left({}^0G\right)^0 \to G$ induces a map $X_\QQ(G)\to X_\QQ\left(\left({}^0G\right)^0\right)$ with finite cokernel and trivial image by equation \eqref{eq:0G0}. This concludes the proof.
\end{proof}

Now let $G$ be an arbitrary connected isotropic linear algebraic group over $\QQ$ with $X_\QQ(G)=1$. Denote by $P$ a minimal parabolic $\QQ$-subgroup of $G$ with unipotent radical $U$. Furthmermore let $K$ be a maximal compact subgroup of $G(\RR)$ and denote by $\theta$ the Cartan involution associate to $K$ in the sense of \cite[Proposition 1.6, Definition 1.7]{borelserre1973}. Then there exists a unique $\theta$-invariant $\QQ$-Levi-subgroup $M$ in $P$ \cite[Corollary 1.9]{borelserre1973}. We denote by $S:=M\cap\mathscr R_{\rm d}(P)$ the maximal $\theta$-invariant $\QQ$-split torus in the radical of $P$.

Finally let $\Phi_\QQ(S,P)\subseteq X_\QQ(S)$ be the set of the relative roots of $P$ with repsect to $S$. Let $\Delta\subseteq\Phi_\QQ(S,P)$ denote the set of simple roots in $\Phi_\QQ(S,P)$. Then $\Delta$ is a basis of the relative root system $\Phi_\QQ(S,G)$ and as such defines an ordering $\geq$ on the latter, such that $\Phi_\QQ(S,P)$ can be identified with the positive roots with respect to $\geq$.

Let $\Gamma\subseteq G(\RR)$ be an arithmetic subgroup and we write $\mathscr X=G(\RR)/K$ for the associated symmetric space. As a reductive Lie group $G(\RR)$ possesses an Iwasawa decomposition \cite[Section 1.11]{borelharishchandra1962}. In order to study arithmetic quotients, we modify this decomposition as in \cite{book_harishchandra1968,borelserre1973,borel1974}. The decomposition we choose is compatible with the left action of $G(\RR)$ on $\mathscr X$, contrary to the right action in the works cited above.

$A:=S(\RR)^0$ is $\theta$-invariant. Define for $t>0$
$$
A_t:=\{a\in A\mid \forall \alpha\in\Delta:\;\alpha(a)\geq t\}.
$$
According to Proposition \ref{prop:boreleserre} we have as in \cite[Section 4.2]{borel1974}
$$
P(\RR)={}^0P(\RR)\rtimes A.
$$
Furthermore
$$
{}^0P=U\rtimes {}^0M,
$$
and as a consequence we get the decomposition
\begin{equation}
U(\RR)\left({}^0M(\RR)\right)A K= G(\RR).
\label{eq:iwasawalanglands}
\end{equation}
Here $a\in A$ is uniquely determined through $g\in G(\RR)$ and the map $g\mapsto a$ is real analytic \cite[Proposition 1.5]{borelserre1973}.
$$
K_P:=K\cap P(\RR)=K\cap{}^0M(\RR)
$$
is maximal compact in $P(\RR)$ and in ${}^0M(\RR)$ as well. With
$$
Z:={}^0M(\RR)/K_P
$$
we get a canonical diffemorphism
$$
{}^0P(\RR)/K_P\cong U(\RR)\times Z.
$$
Now $K$ defines a point $o$ in $\mathscr X$ (we may choose any point fixed by $K$) and we end up with an isomorphism
$$
\mu_o:Y:=U(\RR)\times Z\times A\to\mathscr X,
$$
$$
(u,z,a)\mapsto uza\cdot o.
$$
In the sequel we identify $U(\RR)\times Z\times A$ with its image under $\mu_o$.

\begin{definition}[{\cite[section 3.2]{borel1981}}]\label{defi:borel}
A $\phi\in\mathscr C^\infty(\mathscr X)$ is of {\em moderate growth} if for any compact set $\omega\subseteq U(\RR)\times Z$ and any $t>0$ there exist $C>0$ and $\lambda\geq 0$ such that
\begin{equation}
\forall (x,a)\in\omega\times A_t:\;\;\;
|\phi(xa)|\leq C\cdot\lambda(a).
\label{eq:siegelschranke}
\end{equation}
If for any $\omega,t$ and $\lambda\in X_\QQ(S)$ we may find a $C>0$ such that \eqref{eq:siegelschranke} is fulfilled, then $\phi$ is {\em fast decreasing}.
\end{definition}

$\theta$ defines a Cartan decomposition
$$
\Lie(G(\RR))=:\lieg=\liek\oplus\liep.
$$
More precisely let $\liek$ (resp. $\liep$) be the $1$- (resp. $(-1)$-) eigenspace of the action of $\theta$ on $\lieg$. We write $d:=\dim_\RR\liep=\dim_\RR\mathscr X$ and we choose a basis $\omega_1,\dots,\omega_d$ of the invariant 1-forms $\liep^*$ consisting of Maurer-Cartan forms. Furthermore we have for any subset
$$
I=\{i_1,\dots,i_r\}\subseteq \{1,\dots,d\}
$$
of cardinality $r$ an $r$-form
$$
\omega_I:=\omega_{i_1}\wedge\cdots\wedge\omega_{i_r},
$$
where $i_1< i_2 <\dots <i_r$. The form
$$
\eta=\sum_{I}\omega_I\otimes \phi_I\in\Omega^r(\mathscr X)
$$
is of {\em moderate growth} (resp. {\em fast decreasing}) if all $\phi_I$ share this property.

Following Borel we let $\Omega_{\rm mg}^\bullet(\Gamma\backslash\mathscr X)$ (resp. $\Omega_{\rm fd}^\bullet(\Gamma\backslash\mathscr X)$) denote the subcomplex of $\Omega^\bullet(\Gamma\backslash\mathscr X)$ consisting of forms $\eta$ which are, together with their differential $d\eta$, of moderate growth (resp. fast decreasing).

We call a function $\phi$ on $G(\RR)$ {\em fast decreasing} (resp. {\em of moderate growth}), if there is a $\QQ$-morphism $\rho:G\to\GL_n$ with finite kernel, such that for all $\mu\in\ZZ$ (resp. for one $\mu\geq 0$) we may find a $C>0$, such that for any $x\in G(\RR)$
$$
|\phi(x)|\;\leq\; C\cdot \absnorm{x}^\mu
$$
with
$$
\absnorm{x} := \tr(\rho(x)^t\cdot \rho(x))^{\frac{1}{2}}.
$$
It is well known (cf. \cite[Proposition 3.10]{borel1981}, \cite[\S3, Lemmas 5 and 6]{book_harishchandra1968}) that a $\phi$ on $\Gamma\backslash\mathscr X$ is of moderate growth (resp. fast decreasing) if and only if the same applies to its pullback on $G(\RR)$. In particular these growths conditions are actually intrinsic and do not depend on our choices in Definition \ref{defi:borel}.

We are now ready to define our modular symbols. Let $H$ be a (connected) reductive group over $\QQ$. Let $T$ be the maximal $\QQ$-split central torus in $H$. Then Proposition \ref{prop:boreleserre} says that
$$
H(\RR)={}^0H(\RR)\times T(\RR)^0.
$$
We write $\mathscr Y$ for the symmetric space of $H(\RR)$ associated to a maximal compact subgroup $K'$, so that we get the canonical decomposition
$$
\mathscr Y={}^0\mathscr Y\times T(\RR)^0,
$$
where
$$
{}^0\mathscr Y:={}^0H(\RR)/K',
$$
since $K'$ is contained in ${}^0H(\RR)$ (cf. Proposition \ref{prop:boreleserre}). Appealing to the same proposition we see that we get a decomposition
$$
\Gamma'\backslash\mathscr Y=\Gamma'\backslash{}^0\mathscr Y\times\RR^r
$$
for any arithmetic subgroup $\Gamma'$ of $H(\RR)$.

\begin{proposition}\label{prop:faserint}
Let $G$ and $H$ be connected reductive groups over $\QQ$ with $X_\QQ(G)=1$, $\Gamma$ and $\Gamma'$ arithmetic subgroups of $G(\RR)$ and $H(\RR)$ and $s:H\to G$ be a central morphism with $\RR$-anisotropic kernel.
 Let furthermore $\mathscr X$ and $\mathscr Y$ be associated symmetric spaces and fix $g\in G(\QQ)$, such that $s$ induces a map
$$
\mathscr Y\to\mathscr X,\;\;\;y\mapsto s(y)
$$
of symmetric spaces which itself induces a map
$$
s_g:\Gamma'\backslash\mathscr Y\to\Gamma\backslash\mathscr X,\;\;\;
\Gamma'y\mapsto \Gamma g\cdot s(y)
$$
on arithmetic quotients.

Then for any $\eta\in\Omega_{\rm fd}^q(\Gamma\backslash\mathscr X)$ we may integrate $\delta(s_g)(\eta)$ along the fibers of
$$
\Gamma'\backslash\mathscr Y=\Gamma'\backslash{}^0\mathscr Y\times\RR^r,
$$
where $r=\rang_\QQ\zentrum(Y)$. This integration yields a form
$$
s_{g,*}(\eta)\;\in\;\Omega_{\rm mg}^{q-r}(\Gamma'\backslash{}^0\mathscr Y),
$$
and the map $s_{g,*}$ is a chain map.
\end{proposition}

\begin{proof}
We keep the notation of the preceeding paragraph. With $f$ also $h\mapsto f(gh)$ is fast decreasing and since the $1$-forms $\omega_1,\dots,\omega_d$ are left-invariant, we may assume that $g=1$.

$K\cap s(H(\RR))$ is maximal compact in $s(H(\RR))$. We conclude that the Cartan involution $\theta'$ to $K'$ and $\theta$ induce Cartan involutions on the image of $s$ which actually coincide \cite[Proposition 1.6]{borelserre1973}. Hence $s$ is compatible with $\theta$ and $\theta'$, i.e. $\theta\circ s=s\circ\theta'$.

$s(T)$ is contained in a maximal $\QQ$-split torus $T'$ of $H$. Now $T'$ is contained in a minimal parabolic $\QQ$-subgroup $P$ of $G$. Due to the $\theta'$-invariance of $T(\RR)$ we see that $s(T(\RR))$ is $\theta$-invariant and therefore contained in the unique $\theta$-invariant Levi-subgroup $M/\QQ$ of $P$. Then $S:=M\cap\mathscr R_{\rm d}(P)$ is a maximal $\QQ$-split $\theta$-invariant torus in the radical of $P$ containing $s(T)$.

Remember the Cartan decomposition $\lieg=\liek\oplus\liep$ and fix similarly the Cartan decomposition $\lieh=\liel\oplus\lieq$ of $\Lie(H(\RR))$. As before, choose a basis $\omega_1,\dots,\omega_d$ of $\liep^*$ of Maurer-Cartan forms and similarly a Maurer-Cartan basis $\omega_1',\dots,\omega_{d'}'$ for $\lieq^*$. We may assume that for $1\leq i\leq d'$
$$
\delta(s)(\omega_i)=L(s)^*(\omega_i)=\omega_i'
$$
and for $d'<j\leq d$
$$
\delta(s)(\omega_j)=L(s)^*(\omega_j)=0,
$$
since $L(s)$ is injective. Now let
$$
\eta=\sum_{I}\omega_{I}\otimes f_{I}\in\Omega_{\rm fd}^q(\Gamma\backslash\mathscr X).
$$
Then
$$
\delta(s)(\eta)=\sum_{I'}\omega_{I'}\otimes (\phi_{I'}\circ s),
$$
where $I'$ runs throuh the same subsets of $\{1,\dots, d'\}$ as $I$ does. We fix integration parameters $a_1^{-1}da_1,\dots,a_r^{-1}da_r$, corresponding to $\omega_{d-r+1},\dots,\omega_{d}$ and parametrize the fibers by means of a dual basis $\Delta'^*\subseteq\Phi_\QQ(T,H)^*$. Integration of $\delta(s)(\eta)$ along the fibers then means that we integrate all $\phi_{I'}\circ s$ where $\{d'-r+1,\dots,d'\}\subseteq I'$.

Assume that $I'$ satisfies this condition. For any $h\in H(\RR)$ we have to consider the integral
$$
\int_{\RR_{>0}^r}\phi_{I'}(s(h)s(\lambda_1'^*(a_1)\cdots\lambda_r'^*(a_r)))\frac{da_1}{a_1}\cdots \frac{da_r}{a_r},
$$
where $\lambda_1'^*,\dots,\lambda_r'^*\in\Delta'^*$ run through the corresponding basis of cocharacters of $T$. We have
$$
s(\lambda_1'^*(a_1)\cdots\lambda_r'^*(a_r))\in S(\RR)^0
$$
and since $s$ induces an isogeny of $T$ onto a subtorus $S'$ of $S$, we get, modulo orientation and the finite kernel of this isogeny, the integral
$$
\int_{\RR_{>0}^r}\phi_{I'}(s(h)\lambda_1^*(a_1)\cdots \lambda_r^*(a_r))\frac{da_1}{a_1}\cdots \frac{da_r}{a_r},
$$
where $\RR_{>0}^r$ paramtrizes $S'(\RR)^0$ via the dual simple roots $\lambda_1^*,\dots,\lambda_r^*\in(\Phi_\QQ(S',P)\cap\Delta)^*$.

Let
$$
a(h):=\mu_o^{-1}(s(h)K)\in S(\RR)^0
$$
be the projection of $s(h)$ onto the factor $A=S(\RR)^0$ of the decomposition \eqref{eq:iwasawalanglands}.

By our hypothesis $\phi_{I'}$ is fast decreasing, therefore we find for any $\alpha_1,\dots,\alpha_r\in\ZZ$ and any $t>0$ a $C>0$ such that $\forall a_1,\dots,a_r\geq t:$
$$
|\phi_{I'}(s(h)\lambda_1^*(a_1)\cdots\lambda_r^*(a_r))|\leq
C\cdot a_1^{\alpha_1}\cdots a_r^{\alpha_r}\cdot\lambda_1(a(h))^{\alpha_1}\cdots\lambda_r(a(h))^{\alpha_r}.
$$
We may assume that $C$ is constant on a (compact) neighborhood of $h$. Hence our integral is absolutely convergent and bounded as well on the  $2^r$ segments $X_1\times\cdots\times X_{r}$ with $X_i\in\{(0;\lambda_i(a(h))^{-1}],(\lambda_i(a(h))^{-1};\infty)\}$. The resulting map $\check{\phi}_{I'}$ is also bounded and therefore of moderate growth. We end up with a form
$$
s_{g,*}(\eta):=\sum_{\#\check{I}=q-r} \omega_{\check{I}}'\otimes \check{\phi}_{\check{I}\cup\{1,\dots,r\}}\in\Omega_{\rm mg}^{q-r}(\Gamma'\backslash\mathscr Y),
$$
where the sum ranges over the indices such that $\{1,\dots,r\}\cap \check{I}=\emptyset$.

Finally we observe that
$$
d(s_{g,*}(\eta))=s_{g,*}(d\eta),
$$
and this if of moderate growth, concluding the proof of the lemma.
\end{proof}

\begin{remark}
The assumption that $s$ is central simplifies our formulation. It guarantees that the kernel of $s$ is already defined over $\QQ$. The proposition, as well as its proof are still valid word by word, if we only assume that the kernel of $s$ is defined over $\RR$ and $\RR$-anisotropic.
\end{remark}

For later use we prove the following

\begin{lemma}\label{lem:eigentlich}
Assume $G$ reductive over $\QQ$ and $X_\QQ(G)=1$. Let $\Gamma$ be an arithmetic and $K$ be a maximal compact subgroup of $G(\RR)$, $\Gamma^1$ an arithmetic and $K^1$ a maximal compact subgroup of $G^{\ad}(\RR)$ with $\ad(\Gamma)\subseteq\Gamma^1$ and $\ad(K)\subseteq K^1$. Then the map
$$
p:\Gamma\backslash\mathscr X\to\Gamma^1\backslash\mathscr X^1
$$
induced by $\ad$ on the associated arithmetic quotients is proper.
\end{lemma}

\begin{proof}
Note that $\ad(\Gamma)$ is of finite index in $\Gamma^1$ since $\ad(\Gamma)$ is arithmetic as well. Now $G$ is an isogenous image of $G^{\der}\times Z$ with $Z:=\zentrum(G)^0$. On the other hand $G^{\ad}$ is an isogenous image of $G^{\der}$. Hence by defining $\Gamma_1:=\Gamma\cap G^{\der}(\RR)$, $K_1:=K\cap G^{\der}(\RR)$, $\Gamma_2:=\Gamma\cap Z(\RR)$, $K_2:=K\cap Z(\RR)$, we get a canonical map
$$
p': \Gamma_1\backslash G^{\der}(\RR)/K_1\times \Gamma_2\backslash Z(\RR)/K_2\to\Gamma^1\backslash\mathscr X^1,
$$
$$
\left(\Gamma_1 x K_1,\Gamma_2 y K_2\right)\;\mapsto\;\Gamma^1 \ad(xy) K^1.
$$
This naturally factorizes over $\Gamma\backslash\mathscr X$. More precisely we have $p'=p\circ p''$, where $p''$ is surjective and defined analogously.

Now $\Gamma_1\times\Gamma_2$ is an arithmetic subgroup of $G^{\der}\times Z$, hence $p''$ is proper. It remains to see that $p'$ is proper. Obviously the canonical map
$$
p_1': \Gamma_1\backslash G^{\der}(\RR)/K_1\to\Gamma^1\backslash\mathscr X^1
$$
is proper. Finally we show that $\Gamma_2\backslash Z(\RR)$ is compact, which will conclude the proof. Our argument is nothing but a natural generalization of Dirichlet's classical unit theorem. $Z$ decomposes over $\RR$ into an almost product of an $\RR$-split torus $T_{\rm d}$ and an $\RR$-anistropic torus $T_{\rm a}$. Now $T_{\rm a}(\RR)$ is compact and $\Gamma_{\rm d}:=\Gamma_2\cap T_{\rm d}(\RR)$ is a discrete subgroup in $T_{\rm d}(\RR)$. Our assumption that $X_\QQ(G)=1$ implies $X_\QQ(Z)=1$, hence $\Gamma_2\backslash Z(\RR)$ has finite invariant measure. Therefore $\Gamma_{\rm d}\backslash T_{\rm d}(\RR)$ has finite invariant measure, so $\Gamma_{\rm d}$ is a lattice in $T_{\rm d}(\RR)$ and $\Gamma_{\rm d}\backslash T_{\rm d}(\RR)$ is compact. This proves the lemma.
\end{proof}

\begin{remark}
We get an analogous statement by replacing $G^{\ad}(\RR)^0$ by an arbitrary real Lie group $G'$, such that $G(\RR)^0\to G^{\ad}(\RR)^0$ factorizes over $G'$ and $G(\RR)^0\to G'$ is an epimorphism.
\end{remark}

We may regard $\Gamma'\backslash{}^0\mathscr Y$ as an arithmetic quotient of $\left({}^0G\right)^0$. Hence Proposition \ref{prop:0g0} implies that it has finite invariant measure and by \cite[Theorem 5.2]{borel1981} the canonical inclusion $\Omega_{\rm c}^\bullet\to\Omega_{\rm fd}^\bullet$ induces an isomorphism in cohomolgy. Proposition \ref{prop:faserint} and Lemma \ref{prop:faserint} may be used to construct many different modular symbols. With our application in mind we will stick only to one case. We write $p:{}^0\mathscr Y\to\mathscr Y^1$ for the projection induced by $\ad$, where $\mathscr Y^1$ is a symmetric space for $H^{\ad}(\RR)$. We define the following {\em modular symbol}
$$
\mathscr P_{s,g}^{q}:H_{\rm c}^{q}(\Gamma\backslash\mathscr X,\CC)\times H_{\rm c}^{\dim\mathscr Y-q}(\Gamma^1\backslash\mathscr Y^1,\CC)\to\CC,
$$
$$
([\eta],[\eta'])\;\mapsto\;
\int_{\Gamma'\backslash{}^0\mathscr Y}s_{g,*}(\eta)\wedge p^*(\eta').
$$
Due to its topological construction this pairing respects the natural $\ZZ$-structure on cohomology.

\section{A general local Birch lemma}\label{sec:lokalesbirchlemma}

Fix a nonarchimedean local field $F$. So $F$ is a finite extension of the $p$-adic prime field $\QQ_p$ or of $\FF_p((T))$. Denote by $\OO_F$ its maximal compact subring, by $\mathfrak{p}$ its maximal ideal and by $v$ the place associated to $F$. We fix a prime $\varpi\in\OO_F$ and the absolute value $\absnorm{\cdot}$ of $F$ such that $\absnorm{\varpi}=\absNorm(\mathfrak{p})^{-1}$, where $\absNorm(\mathfrak{a})$ is the absolute norm of a fractional ideal in a local (or a global) field. Fix a character $\psi:F\to\CC^\times$ with conductor $\OO_F$. For a generic automorphic representation $\pi_v$ of $\GL_n(F)$ we write $\mathscr{W}(\pi_v,\psi)$ for the $\psi$-Whittaker space of $\pi_v$. The same notation applies to global Whittaker spaces of global representations. The choice of $\psi$ also fixes the Gau\ss{} sum
$$
G(\chi):=\sum_{x+\mathfrak{f}\in(\OO_F/\mathfrak{f})^\times}\chi(x)\psi\left(\frac{x}{f}\right).
$$
for any quasi-character $\chi:F^\times\to\CC^\times$ of conductor $\mathfrak{f}=f\OO_F$. By abuse of notation we also write $\chi(g)$ instead of $\chi(\det(g))$ for $g\in\GL_n(F)$. We will make repeated use of the following elementary fact. Let $0\neq g\in\OO_F$, then with $\mathfrak{h}:=\mathfrak{f}\cap g\OO_F$ we have
\begin{equation}
\sum_{x+\mathfrak{h}\in(\OO_F/\mathfrak{h})^\times}\chi(x)\psi\left(\frac{x}{g}\right)\;=\;
\begin{cases}
\chi(g/f)\cdot G(\chi),&\text{if\;}\mathfrak{f}=g\OO_F,\\
0,&\text{otherwise}.
\end{cases}
\label{eq:gaussnull}
\end{equation}

We denote by $I_n$ the Iwahori subgroup of $\GL_n(\OO_F)$, i.e. the group of matrices $g\in\GL_n(\OO_F)$ that become upper triangular modulo $\mathfrak{p}$. $B_n$ is the standard Borel subgroup of $\GL_n$ of upper triangular matrices, $U_n$ denotes its unipotent radical. We extend $\psi$ to $U_n(F)$ by the rule
$$
\psi(u):=\prod_{i=1}^{n-1}\psi(u_{ii+1})
$$
for $u=(u_{ij})\in U_n(F)$. We write $B_n^-$ for the subgroup of $\GL_n$ of lower triangular matrices. 

$W_n$ denotes the Weyl group of $\GL_n$ realized as the subgroup of permutation matrices. To each $\omega\in W_n$ we associate a permutation $\sigma\in S_n$ ($S_n$ is the symmetric group on $\{1,\dots,n\}$) via
$$
\omega\cdot b_k = b_{\sigma^{-1}(k)}
$$
for $1\leq k\leq n$. Here $b_k$ denotes the $k$-th standard basis vector of $F^n$. Then for any $a=(a_i)_{1\leq i\leq n}\in F^n$ we have $\omega a=(a_{\sigma(i)})_{1\leq i\leq n}$. The map $\omega\mapsto\sigma^{-1}$ is an isomorphism $W_n\to S_n$. Let $w_n\in W_n$ denote the longest element. We have
$$
w_{n}=
\begin{pmatrix}
&& 1\\
&\adots&\\
1&&
\end{pmatrix}.
$$
For $e=(e_1,\dots,e_n)\in\ZZ^n$ we define the matrix
$$
\pi^e:=\diag(\pi^{e_1},\dots,\pi^{e_n}).
$$

Following Iwahori-Matsumoto \cite[Proposition 2.33]{iwahorimatsumoto1965} and Satake \cite[section 8.2]{satake1963} we get the disjoint decomposition
\begin{equation}
\GL_n(F)=\bigsqcup_{
\begin{subarray}{c}
\omega\in W_n\\e\in\ZZ^n
\end{subarray}
}U_n(F)\varpi^e\omega I_n.
\label{eq:iwasawa}
\end{equation}
Consequently we find for any $g\in \GL_n(F)$ elements $a\in U_n(F),e\in\ZZ^n,\omega\in W_n,s\in I_n$ that satisfy
$$
g=a\cdot \varpi^e\cdot\omega\cdot s.
$$
Then $e$ and $\omega$ are uniquely determined by $g$, but $a$ and $s$ are not.

We fix the Haar measure $dg$ on $\GL_n(F)$ such that the maximal compact subgroup $\GL_n(\OO_F)$ has measure $1$.

Fix integers $m\geq n\geq 0$ and let
$$
j:\GL_n\to\GL_m
$$
be the inclusion
$$
g\mapsto\begin{pmatrix}g&0\\0&{\bf1}_{m-n}\end{pmatrix}.
$$
For a fixed $0\neq f\in\OO_F-\OO_F^\times$ we define the matrix
$$
D_n\;:=\;\diag(f^{-(n-1)},f^{-(n-3)},\dots,f^{n-3},f^{n-1})\;\in\;\GL_n(F)
$$
and the linear form
$$
\lambda_n:F^{n\times n}\to F,\;\;\;g\mapsto b_n^t\cdot g\cdot \phi_n,
$$
where
$$
\phi_n := (f^{-n},f^{-(n-1)},\dots,f^{-1})^t.
$$
The main result of this section is
\begin{theorem}[general local Birch lemma]\label{thm:localbirchlemma}
Let $w$ and $v$ be Iwahori invariant $\psi$- resp. $\psi^{-1}-$Whittaker functions on $\GL_{m}(F)$ and $\GL_n(F)$. For any quasi-character $\chi:F^\times\to\CC^\times$ with nontrivial conductor $\mathfrak{f}=f\OO_F$ we have the explicit formula
$$
\int_{U_{n}(F)\backslash{}\GL_{n}(F)}
\psi(\lambda_n(g))
w\left(
j(g D_n w_n)
\right)
v(g)
\chi(\det(g))
\absnorm{\det(g)}^{s-\frac{m-n}{2}}dg=
$$
$$
\prod_{\nu=1}^n
\left({1-\absNorm(\mathfrak{p})^{-\nu}}\right)^{-1}
\cdot
\absNorm(\mathfrak{f})^{-\sum_{k=1}^{n}k(n+1-k)}\cdot
G(\chi)^{\frac{n(n+1)}{2}}\cdot
w({\bf1}_{m})\cdot
v({\bf1}_n).
$$
\end{theorem}

In the case $m=n+1$ our formula appears simpler than the generalized local Birch lemma of \cite{kazhdanmazurschmidt2000} and contains \cite{schmidt1993,schmidt2001} as a special case. Our proof differs substantially from the previous approches. A crucial observation is that the integrand of the theorem is constant on the classes of a refinement of the Iwasawa decomposition \eqref{eq:iwasawa}. This enables us to evaluate the integral as a sum. It will turn out that for $e\neq 0$ or $\omega\neq{\bf 1}_n$ the corresponding partial sums vanish so that our integral eventually becomes a finite sum (cf. Lemma \ref{lem:zentrales}), giving rise to the formula of the theorem.

For the proof of Theorem \ref{thm:localbirchlemma} we may assume $m=n$ and $j={\rm{id}}$ without loss of generality. The final argument is inductive and we need some preparation for the induction step. We start by giving the refined double coset decomposition of $\GL_n(F)$.

Let
$$
J_{l,n}\;:=\;\kernel\left(\GL_n(\OO_F)\to\GL_n(\OO_F/\mathfrak{f}^{l})\right).
$$
From now on we assume $l\geq 2n$. This guarantees that
\begin{equation}
J_{l,n}\;\;\;\subseteq\;\;\; I_n\cap w_n D_n^{-1} I_n D_n w_n.
\label{eq:jnin}
\end{equation}
Choose a system $R_l$ of representatives of $\OO_F/\mathfrak{f}^l$ and let $R_l^\times\subseteq R_l$ be a system of representatives of $\left(\OO_F/\mathfrak{f}^l\right)^\times$. This enables us to define
$$
\mathfrak{R}_{l,n}:=
\{
(r_{ij})\in I_n
\mid
r_{ij}\in R_{l}
\}.
$$
Then $\mathfrak{R}_{l,n}$ is a system of representatives for $I_n/J_{l,n}$ and as such may be endowed with the natural group structure which is induced by matrix multiplication modulo $\mathfrak{f}^{l}$. We may assume that
\begin{equation}
0,\pm1,\pm f,\dots,\pm f^{l-1}\in R_l
\label{eq:specialrep}
\end{equation}
holds, which simplifies notation. This allows us to define for any $\omega\in W_n$ with corresponding $\sigma\in S_n$
$$
\mathfrak{R}_{l,n}^\omega:=
\{
(r_{ij})\in \mathfrak{R}_{l,n}
\mid
\forall i,j:
i<j
\Rightarrow
r_{\sigma(i)\sigma(j)}=0
\}.
$$
Note that for $r=(r_{ij})_{ij}$ we have
$$
\omega r\omega^{-1}=\left(r_{\sigma(i)\sigma(j)}\right)_{ij}.
$$
Therefore
$$
\mathfrak{R}_{l,n}^\omega\;=\;
\mathfrak{R}_{l,n}\cap \omega^{-1} B_n^-(\OO_F)\omega.
$$
Consequently for any $r\in\mathfrak{R}_{l,n}^\omega$ we see that
\begin{equation}
\det(r)=\prod_{k=1}^n r_{kk}
\label{eq:detr}
\end{equation}
and that furthermore $\mathfrak{R}_{l,n}^{\omega}$ is a subgroup of $\mathfrak{R}_{l,n}$. Our interest in $\mathfrak{R}_{l,n}^{\omega}$ is motivated by

\begin{proposition}\label{prop:repr}
The set $\varpi^e \omega \mathfrak{R}_{l,n}^\omega$ is a system of representatives of the double cosets
$$
U_n(F)\varpi^e\omega s J_{l,n},\;\;\;s\in I_n
$$
in $U_n(F)\varpi^e\omega I_n$. Fix an $l(e)\geq 2n$ for any $e\in\ZZ^n$. Then
$$
\GL_n(F)=\bigsqcup_{
\begin{subarray}{c}
e\in\ZZ^n\\
\omega\in W_n\\
r\in\mathfrak{R}_{l(e),n}^\omega
\end{subarray}
} U_n(F)\varpi^e\omega r J_{l(e),n}.
$$
\end{proposition}

\begin{proof}
Since $\varpi^e U_n(F)\varpi^{-e}=U_n(F)$ we may assume that $e=0$. Define
$$
U^\omega\;:=\;U_n(F)\cap \omega I_n\omega^{-1}.
$$
As a system of representatives for $\omega I_n/J_{l,n}$ the set $\omega\mathfrak{R}_n$ contains a system of representatives for the double cosets
$$
U^\omega\omega s J_{l,n},\;\;\;s\in I_n
$$
in $U^\omega\omega I_n$. We need to show that $\omega\mathfrak{R}_{l,n}^\omega$ is a system of representatives of these double cosets.

For any $r=(r_{ij})\in I_n$ we have
$$
\omega\cdot r\cdot \omega^{-1}=(r_{\sigma(i)\sigma(j)})_{ij}.
$$
The action of $U_n(F)$ from the left allows us to add multiples of a row $i$ with entries $r_{\sigma(i)\sigma(j)}$ to a row $k<i$ with entries $r_{\sigma(k)\sigma(j)}$. Since $r_{\sigma(i)\sigma(i)}$ are units we may use them to annihilate all entries $r_{\sigma(k)\sigma(i)}$ with $k<i$ of the same column. If we proceed inductively, starting with $i=n$, this process corresponds to a multiplication with a matrix
$$
u=(u_{ij})\in U_n(F).
$$
We show inductively that $u\in U^\omega$ and that
$$
u\cdot\omega r\omega^{-1}\in\omega I_n\omega^{-1}\cap B_n^-(\OO_F).
$$
To see this iteratively define matrices $u^{(\nu)}=(u_{ij}^{(\nu)})\in U^\omega$, $r^{(\nu)}\in I_n$ as follows. Let $u^{(0)}:={\bf1}_n$ and $r^{(0)}:=r$. For $\nu\geq 0$ let $u^{(\nu+1)}=(u_{ij}^{(\nu+1)})\in U_n(F)$ be given by $u_{ij}^{(\nu+1)}:=\delta_{ij}$ (Kronecker delta) if $j\neq n-\nu$ or $i\geq j$. For $j=n-\nu$ and $i<j$ let
$$
u_{i,n-\nu}^{(\nu+1)}:=-\frac{r_{\sigma(i)\sigma(n-\nu)}^{(\nu)}}{r_{\sigma(n-\nu)\sigma(n-\nu)}^{(\nu)}}.
$$
Finally define
$$
r^{(\nu+1)}:=\omega^{-1}u^{(\nu+1)}\omega\cdot r^{(\nu)}.
$$
We see inductively that $r^{(\nu)}_{\sigma(i)\sigma(j)}=0$ if $j>n-\nu$ and $i<j$. Furthermore we show that $u^{(\nu)}\in U^\omega$ implies $u^{(\nu+1)}\in U^\omega$. Indeed, the former implies $r^{(\nu)}\in I_n$. The definition of $u^{(\nu+1)}$ shows that there are two cases. If $\sigma(i)<\sigma(n-\nu)$, then $r_{\sigma(i)\sigma(n-\nu)}^{(\nu)}\in\OO_F$, if $\sigma(i)>\sigma(n-\nu)$, then $r_{\sigma(i)\sigma(n-\nu)}^{(\nu)}\in\mathfrak{p}$. It follows that
$$
u_{i,n-\nu}^{(\nu+1)}\in
\begin{cases}
\mathfrak{\OO}_F,&\text{for}\;\sigma(i)<\sigma(n-\nu),\\
\mathfrak{p},&\text{for}\;\sigma(i)>\sigma(n-\nu),
\end{cases}
$$
hence $u^{(\nu+1)}\in U^\omega$.

With
$$
u:=u^{(n-1)}\cdot u^{(n-2)}\cdots u^{(1)}\in U^\omega
$$
we conclude that
$$
u\cdot\omega r\omega^{-1}\in \omega\cdot I_n\cdot\omega^{-1}\cap B_n^-(\OO_F).
$$
Now let $s\in\mathfrak{R}_{l,n}$ be the representative of the double coset
$$
\omega^{-1} u\omega\cdot r\cdot J_{l,n}.
$$
Then $\omega s$ represents the same double coset as $\omega r$ does and
$$
\omega s\omega^{-1}\;\in\;
u\cdot \omega r\omega^{-1}\cdot J_{l,n}
\;\subseteq\;B_n^-(\OO_F) J_{l,n}
$$
yields with \eqref{eq:specialrep} that
$$
s\in \mathfrak{R}_{l,n}^\omega.
$$
Consequently $\omega\mathfrak{R}_{l,n}^\omega$ contains a system of representatives.

To finish the proof, we assume
$$
u\cdot\omega r\;\in\;
\omega s\cdot J_{l,n}
$$
for any $u\in U^\omega$, $r,s\in\mathfrak{R}_{l,n}^\omega$. Since $J_{l,n}$ is normal in $\GL_n(\OO_F)$ this is equivalent to
$$
u\;\in\;
\omega s\omega^{-1}\cdot\omega r^{-1}\omega^{-1}\cdot J_{l,n}.
$$
Now $\omega s\omega^{-1}=(s_{\sigma(i)\sigma(j)})_{ij}$ and $\omega r^{-1}\omega^{-1}$ are lower triangular, so
$$
u\in U_n^-(\OO_F) J_{l,n}.
$$
We observe that this means
$$
u\equiv {\bf1}_n\;(\mmod J_{l,n}),
$$
concluding the proof of the proposition.
\end{proof}

\begin{proposition}\label{prop:ujvolumen}
For any $e\in\ZZ^n$, $\omega\in W_n$ and $r\in\mathfrak{R}_{l,n}^\omega$ the measure
$$
\int_{U_n(\OO_F)\varpi^e\omega r J_{l,n}}dg
$$
is independent of $\omega$ and $r$. If $e=0$ and $l=2n$ we have
$$
\int_{U_n(\OO_F)\omega r J_{2n,n}}dg=
\prod_{\nu=1}^n
\left({1-\absNorm(\mathfrak{p})^{-\nu}}\right)^{-1}
\cdot
\absNorm(\mathfrak{f})^{-n^3-n^2}.
$$
\end{proposition}

\begin{proof}
Since $J_{l,n}$ is normal in $\GL_n(\OO_F)$ we have for $\omega\in W_n$ and $r\in\mathfrak{R}_{l,n}^\omega$
$$
U_n(\OO_F)\varpi^e\omega r J_{l,n}=U_n(\OO_F) \varpi^e J_{l,n}\omega r.
$$
By the right invariance of $dg$ the measure of this set is independent of $r$ and $\omega$.

It is well known that
$$
\left(\GL_n(\OO_F):J_{l,n}\right)=
\#\GL_n(\OO_F/\mathfrak{f}^{l})=
\absNorm(\mathfrak{f})^{ln^2}\cdot
\prod_{\nu=1}^{n}(1-\absNorm(\mathfrak{p})^{-\nu}).
$$
Furthermore $U_n(\OO_F) J_{l,n}$ is a group and
$$
\left(U_n(\OO_F) J_{2n,n}:J_{2n,n}\right)=
\left(U_n(\OO_F):U_n(\OO_F)\cap J_{2n,n}\right)=
$$
$$
\#U_n(\OO_F/\mathfrak{f}^{2n})=
\absNorm(\mathfrak{f})^{2n\cdot\frac{n(n-1)}{2}}=
\absNorm(\mathfrak{f})^{n^3-n^2}.
$$
We finish the proof with the conclusion
$$
\left(\GL_n(\OO_F):U_n(\OO_F) J_{2n,n}\right)=
\prod_{\nu=1}^n
\left({1-\absNorm(\mathfrak{p})^{-\nu}}\right)
\cdot
\absNorm(\mathfrak{f})^{n^3+n^2}.
$$
\end{proof}

The action of the compact torus
$$
T_n:=\left(\OO_F^\times\right)^n
$$
on $\GL_n(F)$, for $\gamma=(\gamma_1,\dots,\gamma_n)\in T_n$ given by
$$
{}^\gamma\cdot:\GL_n(F)\to\GL_n(F),\;\;\;g\mapsto {}^{\gamma}g:=g\cdot \diag(\gamma_1,\dots,\gamma_n)
$$
is vital for our argument. The torus $T_n$ operates naturally on the set of representatives $\mathfrak{R}_{l,n}^\omega$ via its action on the quotient $I_n/J_{l,n}$. This action factorizes via the finite torus
$$
\overline{T}_{l,n}:=T_n/\left(1+\mathfrak{f}^{l}\right)^n
$$
and $\overline{T}_{l,n}$ acts faithfully on the orbit of any $r\in\mathfrak{R}_{l,n}^\omega$.

For the induction step we define under the assumption $\sigma(n)=n$ the set
$$
\tilde{\mathfrak{R}}_{l,n}^\omega:=
\{
(r_{ij})\in\mathfrak{R}_{l,n}^\omega
\mid
r_{n1}= f^{n-1},\;r_{nj}= -f^{n-j},\;2\leq j\leq n
\}.
$$
Note that $\sigma(n)=n$ implies $n\geq 1$.

\begin{proposition}\label{prop:rtildebahn}
If $\sigma(n)=n$ we have for any $r\in\tilde{\mathfrak{R}}_{l,n}^\omega$
$$
\#\left(\overline{T}_{l,n}\cdot r\cap\tilde{\mathfrak{R}}_{l,n}^\omega\right)=
\absNorm(\mathfrak{f})^{\frac{n(n-1)}{2}}.
$$
In other words the orbit of $r$ under the action of $T_n$ on $\mathfrak{R}_{l,n}^\omega$ contains precisely $\absNorm(\mathfrak{f})^\frac{n(n-1)}{2}$ elements of $\tilde{\mathfrak{R}}_{l,n}^\omega$.
\end{proposition}

\begin{proof}
The stabilizer of
$$
(f^{n-1}+\mathfrak{f}^{l},-f^{n-2}+\mathfrak{f}^{l},\dots,-1+\mathfrak{f}^{l})\in \OO_F^n/\left(\mathfrak{f}^{l}\right)^n
$$
under the action $s\mapsto s\cdot{}^\gamma{\bf1}_n$ of $T_n$ on $\OO_F^n/\left(\mathfrak{f}^{l}\right)^n$ is
$$
(1+\mathfrak{f}^{l-n+1})\times(1+\mathfrak{f}^{l-n+2})\times\cdots\times(1+\mathfrak{f}^{l+1})\times(1+\mathfrak{f}^{l}).
$$
Its projection $\overline{T}_{l,n}$ has cardinality
$$
\prod_{j=1}^{n}
\absNorm(\mathfrak{f})^{n-j}=\absNorm(\mathfrak{f})^\frac{n(n-1)}{2}.
$$
This concludes the proof.
\end{proof}

Define the matrix
$$
C_n:=
\begin{pmatrix}
1&0&\cdots&\cdots&0\\
0&\ddots&\ddots&&\vdots\\
\vdots&\ddots&\ddots&\ddots&\vdots\\
0&\cdots&0&1&0\\
f^{n-1}&-f^{n-2}&\hdots&-f&-1
\end{pmatrix}
$$
and the projection
$$
p:F^{n\times n}\to F^{n-1\times n-1},\;\;\;
(g_{ij})\mapsto (g_{ij})_{1\leq i,j\leq n-1}
$$
with its set theoretic section
$$
\tilde{j}:F^{n-1\times n-1}\to F^{n\times n},\;\;\;
\tilde{g}\mapsto
\begin{pmatrix}
\tilde{g} & 0\\
0 & 1
\end{pmatrix}.
$$

\begin{proposition}\label{prop:rtildeindex}
Under the assumption $\sigma(n)=n$ we have for $\tilde{\omega}:=p(\omega)$, $\tilde{r}:=p(r)$,
$$
\#\tilde{\mathfrak{R}}_{l,n}^\omega
=
\#\mathfrak{R}_{l,n-1}^{\tilde{\omega}}.
$$
More precisely the projection $p$ induces a bijection
$$
p:
\tilde{\mathfrak{R}}_{l,n}^\omega\to
\mathfrak{R}_{l,n-1}^{\tilde{\omega}}
$$
and
$$
\GL_{n-1}(F)\to\GL_n(F),\;\;\;
\tilde{g}\mapsto \tilde{j}(\tilde{g})\cdot C_n
$$
induces the inverse of $p$.
\end{proposition}

\begin{proof}
Note that the projection $p: F^{n\times n}\to F^{n-1\times n-1}$ induces a map
$$
p:
\omega\tilde{\mathfrak{R}}_{l,n}^\omega\omega^{-1}\to
\tilde{\omega}\mathfrak{R}_{l,n-1}^{\tilde{\omega}}\tilde{\omega}^{-1}.
$$
Since $\omega\tilde{\mathfrak{R}}_{l,n}^\omega\omega^{-1}$ and $\tilde{\omega}\mathfrak{R}_{l,n-1}^{\tilde{\omega}}\tilde{\omega}^{-1}$ are lower triangular $p$ is well defined, bijective and has the inverse as claimed.
\end{proof}

Let
$$
A_n:=
\begin{pmatrix}
1&f^{-1}&0&\hdots&0\\
0&1&-f^{-1}&\ddots&\vdots\\
\vdots&\ddots&\ddots&\ddots&0\\
\vdots&&\ddots&1&-f^{-1}\\
0&\hdots&\hdots&0&1\\
\end{pmatrix}\in\GL_n(F),
$$
$$
\tilde{A}_n:=
\begin{pmatrix}
f^{-1}&0&\hdots&\hdots&0\\
1&-f^{-1}&\ddots&&\vdots\\
0&1&-f^{-1}&\ddots&\vdots\\
\vdots&\ddots&\ddots&\ddots&0\\
0&\hdots&0&1&-f^{-1}\\
\end{pmatrix}\in\GL_{n}(F),
$$
and
$$
B_{n}:=f\cdot \tilde{A}_{n}\in I_{n}.
$$
Furthermore we let $B_1=C_1={\bf1}_1$ and $B_0:={\bf1}_0$. This guarantees that for all $n\geq 0$
\begin{equation}
\det(B_{n+1}C_{n+1})=\det(B_n).
\label{eq:detbcrel}
\end{equation}

Let $w$ denote an Iwahori invariant $\psi$-Whittaker function on $\GL_{n+1}(F)$ and $j:\GL_{n}\to\GL_{n+1}$ denote the canonical inclusion as before.

\begin{lemma}\label{lem:matrixinduktion}
For $n\geq 0$ and $g\in \GL_n(F)$ we have
$$
w\left(
j(g)
C_{n+1}\cdot D_{n+1} w_{n+1}
\right)\cdot
v(g)=
$$
$$
\psi(\lambda_n(g B_n))\cdot
w\left(
j(g B_n\cdot D_n w_n)
\right)\cdot
v(g B_n).
$$
\end{lemma}

\begin{proof}
Since $B_n\in I_n$ it suffices to show that
$$
\psi(\lambda_n(g B_n))\cdot
w\left(
j(g B_n\cdot D_n w_n)
\right)
=
w\left(
j(g)
C_{n+1}\cdot D_{n+1} w_{n+1}
\right).
$$
To see this consider the matrix
$$
u=(u_{ij})\in U_{n+1}(F),
$$
for $j\leq n$ and $1\leq i\leq n+1$ given by
$$
u_{ij}:=\delta_{ij}\;\;\;\text{(Kronecker delta)}
$$
and for $1\leq i \leq n$ by
$$
u_{in+1}:=-g_{i1}\cdot f^{-n}.
$$
Then, due to $\tilde{A}_{n}=\left((A_{n+1})_{ij+1}\right)_{1\leq i,j \leq n}$, we get
$$
u\cdot j(g) C_{n+1}\cdot A_{n+1}=
\begin{pmatrix}
0&&&\\
\vdots&&g \cdot\tilde{A}_{n}&&\\
0&&&\\
f^{n}&0&\hdots&0
\end{pmatrix}.
$$
Finally we have
$$
u\cdot j(g)C_{n+1}\cdot A_{n+1}\cdot D_{n+1} w_{n+1}=
\begin{pmatrix}
&&&0\\
&\tilde{g}&&\vdots\\
&&&0\\
0&\hdots&0&1
\end{pmatrix}
$$
with
$$
\tilde{g}=g\cdot\tilde{A}_n f\cdot D_{n} w_{n}.
$$
We conclude that
$$
u\cdot j(g)C_{n+1}\cdot A_{n+1}\cdot D_{n+1} w_{n+1}=
j\left(g B_{n}\cdot D_{n} w_{n}\right).
$$
Note that
$$
w_{n+1} D_{n+1}^{-1}\cdot
A_{n+1}\cdot
D_{n+1} w_{n+1}\in I_{n+1},
$$
which implies
$$
j(g)C_{n+1}\cdot A_{n+1}\cdot D_{n+1} w_{n+1}\in j(g)C_{n+1}\cdot D_{n+1} w_{n+1} I_{n+1}.
$$
For our Iwahori invariant $\psi$-Whittaker function this means that
$$
\psi(f^{-n}\cdot g_{n1})\cdot
w\left(
j(g B_n\cdot D_n w_n)
\right)
=
w\left(
j(g)
C_{n+1}\cdot D_{n+1} w_{n+1}
\right).
$$
Thanks to
$$
B_n\phi_n=\left(f^{-n},0,\dots,0\right)^t
$$
the claim follows.
\end{proof}

The next lemma is the key ingredient in the proof of Theorem \ref{thm:localbirchlemma}.
\begin{lemma}\label{lem:zentrales}
Let $w$ and $v$ be Iwahori invariant $\psi-$ (resp. $\psi^{-1}$-) Whittaker functions on $\GL_n(F)$. For any $n\geq 0$, $e\in\ZZ^n$, $\omega\in W_n$ and $l\geq\max\{2n,n-e_1/\nu_\mathfrak{p}(f),\dots,n-e_n/\nu_\mathfrak{p}(f)\}$ we have
$$
\sum_{g\in\varpi^e\omega \mathfrak{R}_{l,n}^\omega}
\!\!\!\psi(\lambda_n(g))\cdot
w(
g\cdot D_n w_n
)\cdot
v(g)\cdot
\chi(\det(g))\cdot
\absnorm{\det(g)}^{s}=
$$
$$
\begin{cases}
\absNorm(\mathfrak{f})^{\frac{(l-2n)n(n+1)}{2}+\frac{1}{2}\sum_{\nu=1}^{n}5\nu^2-3\nu}\cdot
G(\chi)^{\frac{n(n+1)}{2}}\cdot w({\bf1}_{n})\cdot v({\bf1}_n),\hspace*{2.5em}\ \\
\hfill\text{for $\omega={\bf1}_n$ and $e=0$,}\\
0,\hfill\text{otherwise}.
\end{cases}
$$
\end{lemma}

\begin{proof}
We proceed by induction on $n$. If $n=0$, then $W_0=\GL_0(\OO_F)=\{{\bf1}_0\}$, $\mathfrak{R}_0^\omega=\{{\bf1}_0\}$, $\ZZ^0=\{0\}$. The case $\omega\neq{\bf1}_0$ or $e\neq 0$ actually never occurs. This concludes the case $n=0$. Now let $n\geq 1$ and suppose that the claim is true for $n-1$.

Remember that $\overline{T}_{l,n}$ acts faithfully on the orbits of the action of $T_n$ on $\mathfrak{R}_{l,n}^\omega$. Let $S\subseteq T_{n}$ be any system of representatives for $\overline{T}_{l,n}$. The decomposition of $\mathfrak{R}_{l,n}^\omega$ into orbits of this action naturally yields partial sums
$$
Z(r):=
$$
$$
\sum_{\gamma\in S}
\psi(\lambda_n(\varpi^e\omega {}^\gamma r))\cdot
w\left(
\varpi^e\omega {}^\gamma r\cdot D_n w_n
\right)\cdot
$$
$$
v(\varpi^e\omega{}^\gamma r)\cdot
\chi(\varpi^e\omega{}^\gamma r)\cdot
\absnorm{\det(\varpi^e\omega{}^\gamma r)}^{s},
$$
where $r\in\mathfrak{R}_{l,n}^\omega$ represents a $T_n$-orbit. Due to our hypothesis on $l$ this sum is independent of the choice of $S$ and furthermore $Z(r)$ is by definition constant on the $T_n$-orbits. In particular $Z(r)$ vanishes if and only if it vanishes on the whole orbit of $r$.

Our strategy of proof is to see in which cases $Z(r)$ vanishes. This will enable us to deduce equation \eqref{eq:partialsummen}. We will see that this is already enough to conclude the proof inductively by appealing to Lemma \ref{lem:matrixinduktion}.

We have
$$
\varpi^e\omega {}^\gamma r\cdot D_n w_n=
$$
$$
\varpi^e\omega r\cdot{}^\gamma {\bf1}_n\cdot D_n w_n=
$$
$$
\varpi^e\omega r\cdot D_n  w_n\cdot \left(w_n{}^\gamma {\bf1}_n w_n\right)\in
$$
$$
\varpi^e\omega r\cdot D_n  w_n\cdot I_n,
$$
because $w_n{}^\gamma {\bf1}_n w_n\in I_n$. From this relation we get
$$
Z(r)=
$$
$$
\absnorm{\det(\varpi^e)}^{s}\cdot\chi(\varpi^e\omega)\cdot
w\left(
\varpi^e\omega r\cdot D_n w_n
\right)\cdot
$$
$$
v(\varpi^e\omega)\cdot
\chi(r)\cdot
\sum_{\gamma\in S}
\chi\left({}^\gamma{\bf1}_n\right)\cdot
\psi(\lambda_n(\varpi^e\omega{}^\gamma r)).
$$
We have
$$
\psi(\lambda_n(\varpi^e\omega{}^\gamma r))=
\prod_{\nu=1}^n
\psi\left(\varpi^{e_n} f^{\nu-n-1} r_{\sigma(n)\nu}\cdot\gamma_\nu\right),
$$
which yields
$$
\sum_{\gamma\in S}
\chi\left({}^\gamma{\bf1}_n\right)\cdot
\psi(\lambda_n(\varpi^e\omega{}^\gamma r))=
$$
$$
\prod_{\nu=1}^{n}
\sum_{\gamma_\nu\in\left(\OO_F/\mathfrak{f}^{l}\right)^\times}
\chi(\gamma_\nu)\cdot
\psi\left(\varpi^{e_n} f^{\nu-n-1} r_{\sigma(n)\nu}\cdot\gamma_\nu\right).
$$
Since $r_{\sigma(n)\sigma(n)}\in\OO_F^\times$ and $l\nu_\mathfrak{p}(f)\geq n\nu_\mathfrak{p}(f)-e_n$ we conclude with equation \eqref{eq:gaussnull} that we have an implication
\begin{equation}
e_n\neq (n-\sigma(n))\cdot\nu_\mathfrak{p}(f)\;\Rightarrow\;Z(r)=0.
\label{eq:ennminussigman}
\end{equation}
Consequently let $e_n= (n-\sigma(n))\cdot\nu_\mathfrak{p}(f)$. If $\sigma(n)\neq n$, then $e_n>0$ and therefore
$$
\absnorm{\varpi^{e_n} f^{n-n-1} r_{\sigma(n)n}\cdot\gamma_{n}}<\absnorm{f^{-1}},
$$
which implies again that
$$
Z(r)=0.
$$
Therefore we may furthermore assume that $\sigma(n)=n$, from which we immediately get $e_n=0$. Finally we have for any $\nu<n$, because of \eqref{eq:gaussnull}, an implication
\begin{equation}
\absnorm{r_{n\nu}}\neq\absnorm{f^{n-\nu}}\;\Rightarrow\;Z(r)=0.
\label{eq:rnnubetraege}
\end{equation}
With \eqref{eq:specialrep} we can therefore assume that we have $r_{n1}=f^{n-1}$ and
$$
r_{n\nu}=-f^{n-\nu},\;\;\;2\leq\nu\leq n
$$
(in the case $n=1$ we have $r_{nn}=1$), because $Z(r)$ is constant on the $T_n$-orbits and in any orbit with $Z(r)\neq 0$ we may find a representative with the above property. Under these assumptions we get
\begin{equation}
\sum_{\gamma\in S}
\chi\left({}^\gamma{\bf1}_n\right)\cdot
\psi(\lambda_n(\varpi^e\omega{}^\gamma r))=
\chi(B_n)\cdot G(\chi)^n\cdot\absNorm(\mathfrak{f})^{l\cdot n-n}.
\label{eq:gausssummen}
\end{equation}
Since furthermore $Z(r)$ is constant on the orbits, we deduce from Proposition \ref{prop:rtildebahn} that
$$
\sum_{g\in\varpi^e\omega \mathfrak{R}_{l,n}^\omega}
\psi(\lambda_n(g))
w(g\cdot D_n w_n)
v(g)\chi(g)
\absnorm{\det(g)}^{s}=
$$
\begin{equation}
\absNorm(\mathfrak{f})^{-\frac{n(n-1)}{2}}\cdot
\sum_{r\in\tilde{\mathfrak{R}}_{l,n}^\omega}
Z(r).
\label{eq:partialsummen}
\end{equation}

In equation \eqref{eq:ennminussigman} we have already seen that
$$
\sum_{g\in\varpi^e\omega \mathfrak{R}_{l,n}^\omega}
\psi(\lambda_n(g))
w(g\cdot D_n w_n)
v(g)\chi(g)
\absnorm{\det(g)}^{s}=0
$$
for $\sigma(n)\neq n$ or $e_n\neq 0$. Hence we may assume that $\sigma(n)=n$ and $e_n=0$. With a view to the induction step we define $\tilde{e}:=(e_\nu)_{1\leq\nu\leq n-1}$, $\tilde{\omega}:=p(\omega)$ and $\tilde{r}:=p(r)$. Thanks to equations \eqref{eq:detbcrel}, \eqref{eq:gausssummen}, Proposition \ref{prop:rtildeindex} and Lemma \ref{lem:matrixinduktion} we have
$$
\sum_{r\in\tilde{\mathfrak{R}}_{l,n}^\omega}Z(r)=
$$
$$
\chi(B_n)\cdot G(\chi)^n\cdot\absNorm(\mathfrak{f})^{l\cdot n-n}
\cdot
\frac
{\#\tilde{\mathfrak{R}}_{l,n}^\omega}
{\#\mathfrak{R}_{l,n-1}^{\tilde{\omega}}}
\cdot
$$
$$
\sum_{\tilde{r}\in\mathfrak{R}_{l,n-1}^{\tilde{\omega}}}
\!\!\!\!
\absnorm{\det(\varpi^{\tilde{e}})}^{s}\cdot
\!\chi(
\tilde{j}(\varpi^{\tilde{e}}\tilde{\omega}\tilde{r}) C_{n})\cdot
$$
$$
w(\tilde{j}(\varpi^{\tilde{e}}\tilde{\omega}\cdot\tilde{r})C_n\cdot D_n w_n)\cdot v(\tilde{j}(\varpi^{\tilde{e}}\tilde{\omega}))=
$$
$$
G(\chi)^n\absNorm(\mathfrak{f})^{l\cdot n-n}\cdot
$$
$$
\sum_{\tilde{r}\in\mathfrak{R}_{l,n-1}^{\tilde{\omega}}}
\!\!\!\!
\absnorm{\det(\varpi^{\tilde{e}})}^{s}\cdot
\chi(\varpi^{\tilde{e}}
\tilde{\omega}\cdot
\tilde{r} B_{n-1})\cdot
\psi(\lambda_{n-1}(
\varpi^{\tilde{e}}\tilde{\omega}\cdot\tilde{r} B_{n-1}))\cdot
$$
$$
w(\tilde{j}(\varpi^{\tilde{e}}\tilde{\omega}\cdot\tilde{r} B_{n-1}\cdot D_{n-1} w_{n-1}))\cdot v(\tilde{j}(\varpi^{\tilde{e}}\tilde{\omega})).
$$
Multiplication with $B_{n-1}\in I_{n-1}$ only permutes the double cosets, hence
$$
\sum_{r\in\tilde{\mathfrak{R}}_{l,n}^\omega}\!\!Z(r)=
\absNorm(\mathfrak{f})^{l\cdot n-n}\cdot
G(\chi)^n\cdot
$$
$$
\sum_{\tilde{r}\in\mathfrak{R}_{l,n-1}^{\tilde{\omega}}}
\!\!\!\!
\absnorm{\det(\varpi^{\tilde{e}})}^{s}\cdot
\chi(\varpi^{\tilde{e}}
\tilde{\omega}
)\cdot
\psi(\lambda_{n-1}(
\varpi^{\tilde{e}}\tilde{\omega}\tilde{r}))\cdot
$$
$$
w(\tilde{j}(\varpi^{\tilde{e}}\tilde{\omega}\tilde{r}\cdot D_{n-1} w_{n-1}))\cdot v(\tilde{j}(\varpi^{\tilde{e}}\tilde{\omega})).
$$
With the induction hypothesis this yields
$$
\sum_{r\in\tilde{\mathfrak{R}}_{l,n}^\omega}Z(r)=0,
$$
if $e\neq 0$ or $\omega\neq{\bf1}_n$. For $e=0$ and $\omega={\bf1}_n$ it follows that
$$
\absNorm(\mathfrak{f})^{-\frac{n(n-1)}{2}}\cdot
\sum_{r\in\tilde{\mathfrak{R}}_{l,n}^\omega}Z(r)=
$$
$$
G(\chi)^n\cdot
G(\chi)^{\frac{(n-1)n}{2}}\cdot
\absNorm(\mathfrak{f})^{l\cdot n-\frac{n(n+1)}{2}}\cdot
\absNorm(\mathfrak{f})^{\frac{(l-2(n-1))n(n-1)}{2}}\cdot
$$
$$
\absNorm(\mathfrak{f})^{\frac{1}{2}\sum_{\nu=1}^{n-1}5\nu^2-3\nu}\cdot
w({\bf1}_{n})\cdot
v({\bf1}_n)=
$$
$$
\absNorm(\mathfrak{f})^{\frac{(l-2n)n(n+1)}{2}+\frac{1}{2}\sum_{\nu=1}^{n}5\nu^2-3\nu}\cdot
G(\chi)^{\frac{(n+1)n}{2}}\cdot
w({\bf1}_{n})\cdot
v({\bf1}_n),
$$
concluding the proof of Lemma \ref{lem:zentrales}.
\end{proof}

\begin{proof}[Proof of Theorem \ref{thm:localbirchlemma}]
Propositions \ref{prop:repr} and \ref{prop:ujvolumen} show that the twisted local zeta integral of the theorem may be expressed as the finite sum of Lemma \ref{lem:zentrales}. We may choose $l(0)=2n$, so that the elementary formula
$$
\sum_{\nu=1}^{n}\nu(n+1-\nu)=n^3+n^2-\frac{1}{2}\sum_{\nu=1}^{n}(5\nu^2-3\nu)
$$
concludes the proof of the theorem.
\end{proof}

For comparison with earlier work in the case $m=n+1$ \cite{kazhdanmazurschmidt2000,schmidt2001} and also for immediate applications we define
$$
h:=
\begin{pmatrix}
&&&1\\
&w_n&&\vdots\\
&&&\vdots\\
0&\hdots&0&1
\end{pmatrix}\in\GL_{n+1}(\ZZ),
$$
and for any $f\in F^\times$
$$
h^{(f)}:=t^{-1}\cdot h\cdot t,
$$
where
$$
t\;:=\;\diag(f^{n},f^{n-1},\dots,f,1)\;\in\;\GL_{n+1}(F).
$$
We note that for $g=(g_{ij})\in\GL_{n+1}(F)$
$$
t^{-1}\cdot g \cdot t=\left(f^{i-j}\cdot g_{ij}\right)_{ij}.
$$

\begin{corollary}\label{kor:schmidtformel}
Let $n\geq 0$ and choose Iwahori invariant $\psi$- resp. $\psi^{-1}$-Whittaker functions $w$ and $v$ on $\GL_{n+1}(F)$ and $\GL_n(F)$ respectively. Then for any quasi-character $\chi:F^\times\to\CC^\times$ with non-trivial conductor $\mathfrak{f}=f\OO_F$ we have
$$
\int_{U_{n}(F)\backslash{}\GL_{n}(F)}
w\left(
j(g)
\cdot h^{(f)}
\right)\cdot
v(g)\cdot
\chi(\det(g))
\cdot\absnorm{\det(g)}^{s-\frac{1}{2}}dg=
$$
$$
\prod_{\nu=1}^n
\left({1-\absNorm(\mathfrak{p})^{-\nu}}\right)^{-1}
\cdot
\absNorm(\mathfrak{f})^{-\sum_{k=1}^{n}k(n+1-k)}\cdot
G(\chi)^{\frac{n(n+1)}{2}}\cdot
w({\bf1}_{n+1})\cdot
v({\bf1}_n).
$$
\end{corollary}

\begin{proof}
Remember the definition of the matrix $B_n$ and define
$$
E_n:=
\begin{pmatrix}
0&\hdots&\hdots&0&1\\
\vdots&&\adots&\adots&f\\
\vdots&\adots&\adots&\adots&\vdots\\
0&\adots&f&\hdots&f^{n-2}\\
1&-f&-f^2&\hdots&-f^{n-1}
\end{pmatrix}.
$$
Then
$$
E_n^{-1}=
\begin{pmatrix}
0&\hdots&0&f&1\\
\vdots&\adots&-f&\adots&0\\
0&\adots&\adots&\adots&\vdots\\
-f&\adots&\adots&&\vdots\\
1&0&\hdots&\hdots&0
\end{pmatrix},
$$
and
$$
B_n^{-1}=
\begin{pmatrix}
1&0&\hdots&\hdots&0\\
f&-1&\ddots&&\vdots\\
f^2&-f&\ddots&\ddots&\vdots\\
\vdots&\vdots&\ddots&\ddots&0\\
f^{n-1}&-f^{n-2}&\hdots&-f&-1
\end{pmatrix}.
$$
Obviously $B_n$ and $E_n^{-1}$ are related (let $d_i$ denote the $n\times n$ diagonal matrix which differs from the identity only in $(i,i)$ and displays a $-1$ there; then $d_1 B_n d_n E_n d_1=-w_n$). We are interested in these matrices because of the relation
$$
B_{n+1} h^{(f)} E_{n+1}=
\diag\left(f^{-n},f^{-(n-2)},\dots,f^{n-2},f^{n}\right)=D_{n+1}.
$$
Note that conjugation with $D_{n+1}$ equals conjugation with $t^{-2}$. With this notation we have
$$
j(B_n)\cdot B_{n+1}^{-1}=C_{n+1}.
$$
Because of
$$
B_n\in I_n,\;E_{n+1} w_{n+1}\in I_{n+1}
$$
and due to the invariance of the local zeta integral of the corollary with respect to the substitution $g\mapsto  g B_n$ we get
$$
\int_{U_{n}(F)\backslash{}\GL_{n}(F)}
w\left(
j(g)
\cdot h^{(f)}
\right)\cdot
v(g)\cdot
\chi(g)
\cdot\absnorm{\det(g)}^{s-\frac{1}{2}}dg=
$$
$$
\chi(\det(B_n))\cdot
$$
$$
\int_{U_{n}(F)\backslash{}\GL_{n}(F)}
w\left(
j(g)
C_{n+1}\cdot
D_{n+1}\omega_{n+1}
\right)\cdot
v(g)\cdot
\chi(g)
\cdot\absnorm{\det(g)}^{s-\frac{1}{2}}dg.
$$
Lemma \ref{lem:matrixinduktion} shows that this equals
$$
\int_{U_{n}(F)\backslash{}\GL_{n}(F)}
\psi(\lambda_n(g B_n))\cdot
w\left(
j(g B_n\cdot D_n w_n)
\right)\cdot
$$
$$
v(g B_n)\cdot
\chi(g B_n)
\cdot\absnorm{\det(g)}^{s-\frac{1}{2}}dg.
$$
Another substitution $g\mapsto gB_n^{-1}$ yields
$$
\int_{U_{n}(F)\backslash{}\GL_{n}(F)}
\psi(\lambda_n(g))
w\left(
j(g D_n w_n)
\right)
v(g)
\chi(g)
\absnorm{\det(g)}^{s-\frac{1}{2}}dg,
$$
concluding the proof.
\end{proof}

\section{A general global Birch lemma}

In this section $k$ denotes a global field, i.e. a finite extension of $\QQ$ or $\FF_p(T)$ and we fix $j:\GL_{n}\to\GL_m$ as before. Let $\pi$ and $\sigma$ be irreducible cuspidal automorphic representations of $\GL_m(\Adeles_k)$ and $\GL_{n}(\Adeles_k)$ respectively. By $S_\infty$ we denote the set of infinite places of $k$, being empty if $\characteristic k\neq 0$. Let $S$ denote the set of finite places where $\pi$ or $\sigma$ ramifies. Furthermore let $\mathfrak{p}\not\in S$ be a fixed finite place.

For any finite place $\mathfrak{q}$ of $k$ we choose a fixed additive charakter $\psi_\mathfrak{q}:k_\mathfrak{q}\to\CC^\times$ with conductor $\OO_{k_\mathfrak{q}}$. Futhermore we choose non-trivial additive characters for the archimedean completions of $k$ such that $\psi:=\Otimes\limits_{\mathfrak{q}\in M_k} \psi_\mathfrak{q}$ is a character of $k\backslash\Adeles_k$.

We assume $m>n$ for simplicity. For any finite place $\mathfrak{q}$ with residue field cardinality $q$ we have for any pair
$$
(w_\mathfrak{q},v_\mathfrak{q})\in{\mathscr{W}}(\pi_\mathfrak{q},\psi_\mathfrak{q})\times{\mathscr{W}}(\sigma_\mathfrak{q},\psi_\mathfrak{q}^{-1})
$$
of Whittaker functions the local zeta integral
\begin{equation}
\Psi(w_\mathfrak{q},v_\mathfrak{q},s):=
\int_{U_{m}(k_\mathfrak{q})\backslash{}\GL_{m}(k_\mathfrak{q})}
w_\mathfrak{q}
(j(g))
v_\mathfrak{q}(g)
\absnorm{\det(g)}_\mathfrak{q}^{s-\frac{m-n}{2}}dg.
\label{eq:lokalesrankinselbergintegral}
\end{equation}
This integral converges absolutely for ${\rm Re}(s)\gg 0$, \cite{jpss1979a,jpss1979b,jpss1983}. More precisely $\Psi(w_\mathfrak{q},v_\mathfrak{q},s)$ has a meromorphic continuation on $\CC$ and is eventually a non-zero rational function in $q^{-s}$. Finally the collection of these integrals spans a fractional ideal in $\CC(q^{s})$ with respect to the subring $\CC[q^{-s},q^{s}]$. Any generator $T(s)$ of this ideal is of the form
$$
T(s)=P(q^{-s})^{-1}
$$
with a polynomial $P(X)\in\CC[X]$. The local $L$-function $L(s,\pi_\mathfrak{q}\times\sigma_\mathfrak{q})$ is defined as the unique generator for which $P(0)=1$ holds \cite{jpss1983}. In general $L(s,\pi_\mathfrak{q}\times\sigma_\mathfrak{q})$ is not of the form $\Psi(w_\mathfrak{q},v_\mathfrak{q},s)$, but only a finite sum of local zeta integrals. In any case we find $L(s,\pi_\mathfrak{q}\times\sigma_\mathfrak{q})$ in the image of the map
$$
\Psi:{\mathscr{W}}(\pi_\mathfrak{q},\psi_\mathfrak{q})\otimes{\mathscr{W}}(\sigma_\mathfrak{q},\psi_\mathfrak{q}^{-1})\to\CC(q^{s}),\;\;\;
w_\mathfrak{q}\otimes v_\mathfrak{q}\mapsto \Psi(w_\mathfrak{q},v_\mathfrak{q},s).
$$
Therefore there exists always a {\em good tensor} $t_\mathfrak{q}^0\in{\mathscr{W}}(\pi_\mathfrak{q},\psi_\mathfrak{q})\otimes{\mathscr{W}}(\sigma_\mathfrak{q},\psi_\mathfrak{q}^{-1})$ such that
$$
L(s,\pi_\mathfrak{q}\times\sigma_\mathfrak{q})=\Psi(t_\mathfrak{q}^0,s).
$$
If $\pi_\mathfrak{q}$ and $\sigma_\mathfrak{q}$ are unramified, we may choose $t_\mathfrak{q}^0=w_\mathfrak{q}^0\otimes v_\mathfrak{q}^0$ with the corresponding new vectors $w_\mathfrak{q}^0$ and $v_\mathfrak{q}^0$. By Shintani's explicit formula \cite{shintani1976} (cf. \cite[section 2]{jacquetshalika1981a}) we then have
$$
L(s,\pi_\mathfrak{q}\times\sigma_\mathfrak{q})=\det({\bf 1}_{mn}-\absNorm(\mathfrak{q})^{-s}A_{\pi_\mathfrak{q}}\otimes A_{\sigma_\mathfrak{q}})^{-1},
$$
for any place $\mathfrak{q}\not\in S\cup S_\infty$, where $A_{\pi_\mathfrak{q}}$ and $A_{\sigma_\mathfrak{q}}$ denote the corresponding Satake parameters ($\pi$ and $\sigma$ are always generic \cite{shalika1974}).

Now let $(w,v)\in\mathscr{W}_0(\pi,\psi)\times\mathscr{W}_0(\sigma,\psi^{-1})$ be a pair of global Whittaker functions with factorizations $w=\Otimes\limits_{\mathfrak{q}} w_\mathfrak{q}$, $v=\Otimes\limits_{\mathfrak{q}} v_\mathfrak{q}$. Then by Fourier transform we have associated automorphic forms $\phi$ on $\GL_m(\Adeles_k)$ and $\varphi$ on $\GL_{n}(\Adeles_k)$ respectively (cf. \cite{cogdellpiatetskishapiro2004} for example). Furthermore we have a projection $\PP_n^m$ from the space of cuspidal automorphic forms on $\GL_m(\Adeles_k)$ to the space of cuspidal functions on $P_{n+1}(\Adeles_k)$, the standard mirabolic subgroup $P_{n+1}\subseteq\GL_{n+1}$. For ${\rm Re}(s)\gg 0$ the Euler product
$$
\prod_{\mathfrak{q}}\Psi(w_\mathfrak{q},v_\mathfrak{q},s)=
\int_{\GL_{n}(k)\backslash\GL_{n}(\Adeles_k)}
\PP_n^m\phi
\left(
\begin{pmatrix}
g & \\
  & 1
\end{pmatrix}
\right)
\varphi(g)\absnorm{\det(g)}^{s-\frac{1}{2}}dg
$$
converges absolutely and has an analytic continuation to $\CC$, because the global Rankin-Selberg integral on the right hand side is entire \cite[Proposition 6.1]{cogdellpiatetskishapiro1994}, \cite[Section 3.3]{jacquetshalika1981b}. As in the local setting the map
$$
(\Otimes\limits_{\mathfrak{q}} w_\mathfrak{q})\otimes (\Otimes\limits_{\mathfrak{q}} v_\mathfrak{q})\mapsto[s\mapsto \prod_{\mathfrak{q}}\Psi(w_\mathfrak{q},v_\mathfrak{q},s)]
$$
induces a $\CC$-linear map on the algebraic tensor product ${\mathscr{W}}_0(\pi,\psi)\otimes{\mathscr{W}}_0(\sigma,\psi^{-1})$. In the image of this map we find the global $L$-function
$$
L(s,\pi\times\sigma)=\prod_\mathfrak{q} L(s,\pi_\mathfrak{q}\times \sigma_\mathfrak{q}),
$$
modulo the Gamma factors (if $k$ is a number field). More precisely by \cite{jacquetshalika1990s} for any choice of $(w_\mathfrak{q},v_\mathfrak{q})$ for $\mathfrak{q}\in S_\infty$ there is an entire function $P$, such that
$$
P(s)\cdot L(s,\pi\times\sigma)=
\prod_{\mathfrak{q}\in S_\infty}\Psi(w_\mathfrak{q}, v_\mathfrak{q},s)\cdot
\prod_{\mathfrak{q}\not\in S_\infty} \Psi(t_\mathfrak{q}^0,s),
$$
where $t_\mathfrak{q}^0$ are good tensors as before. We know that $P(s)$ is a product of local integrals that depend on the choice of $(w_\mathfrak{q},v_\mathfrak{q})$ for archimedean $\mathfrak{q}\in S_\infty$. Furthermore for any $s_0$ there is a choice of $(w_\mathfrak{q},v_\mathfrak{q})$ for $\mathfrak{q}\in S_\infty$, such that so $P(s_0)\neq 0$ \cite[Theorem 1.2]{cogdellpiatetskishapiro2004}. In particular, at least in the case $m=n+1$ the local $L$-functions at infinity are given by finite sums of Rankin-Selberg integrals as well \cite[Theorem 1.3]{cogdellpiatetskishapiro2004}. The question wether $P(\frac{1}{2})\neq 0$ is intimately related to the problem if $\Omega\neq 0$. In the function field case there is no similar problem, we may assume $P\equiv 1$ in this case. This allows us to give a uniform treatment for all global field $k$.

Finally we have a representation
$$
\Otimes\limits_{\mathfrak{q}\in S_\infty}(w_\mathfrak{q}\otimes v_\mathfrak{q})\otimes\Otimes\limits_{\mathfrak{q}} t_\mathfrak{q}^0=\sum_{\iota}w_\iota\otimes v_\iota,
$$
where any $w_\iota\otimes v_\iota$ is a product of pure tensors. With the corresponding associated automorphic forms $(\phi_\iota,\varphi_\iota)$ we get the integral representation
$$
P(s)\cdot L(s,\pi\times\sigma)=
$$
$$
\sum_{\iota}\int_{\GL_{n}(k)\backslash\GL_{n}(\Adeles_k)}
\PP_n^m\phi_\iota
\left(
\begin{pmatrix}
g & \\
  & 1
\end{pmatrix}
\right)
\varphi_\iota(g)\absnorm{\det(g)}^{s-\frac{1}{2}}dg.
$$

In order to study the twisted $L$-function $L(s,(\pi\otimes\chi)\times\sigma)$ for a quasi-character $\chi$ with $\mathfrak{p}$-power conductor $\mathfrak{f}$ we modify the local Whittaker functions at $\mathfrak{p}$ and allow Iwahori invariant pairs only. This will enable us to apply the local Birch lemma to prove

\begin{theorem}[general global Birch lemma]\label{thm:globalbirchlemma}
Let $\chi$ be any quasi-character of the id\`ele class group of $k$ with non-trivial conductor $\mathfrak{f}$ and trivial at infinity. Then for any choice of $(w_\mathfrak{q},v_\mathfrak{q})$ for $\mathfrak{q}\in S_\infty$ and for any Iwahori invariant pair $(w_\mathfrak{p},v_\mathfrak{p})$ we have with the corresponding entire function $P$,
$$
P(s)
\delta(w_{\mathfrak{p}},v_{\mathfrak{p}})
\chi(-1)^{n+1}
G(\chi)^{\frac{n(n+1)}{2}}
\absNorm(\mathfrak{f})^{-\!\!\sum_{k=1}^{n}k(n+1-k)}
L(s,(\pi\otimes\chi)\times\sigma)=
$$
$$
\sum_{\iota}
\int_{\GL_{n}(k)\backslash\GL_{n}(\Adeles_k)}
\PP_n^m
\phi_\iota
\left(
\begin{pmatrix}
g&\\&1
\end{pmatrix}
h^{(f)}
\right)
\varphi_\iota(g)
\chi(\det(g))
\absnorm{\det(g)}^{s-\frac{1}{2}}
dg,
$$
where
$$
\delta(w_\mathfrak{p},v_\mathfrak{p}):=
w_{\mathfrak{p}}({\bf1}_m)
\cdot
v_{\mathfrak{p}}({\bf1}_{n})
\cdot
\prod_{\nu=1}^{n}\left(1-\absNorm(\mathfrak{p})^{-\nu}\right)^{-1}.
$$
\end{theorem}

\begin{proof}
Any good tensor $t_\mathfrak{q}^0$ for $(\pi_\mathfrak{q},\sigma_\mathfrak{q})$ gives rise to a good tensor $\chi_\mathfrak{q}(\det)\cdot t_\mathfrak{q}^0$ for $(\pi_\mathfrak{q}\otimes\chi_\mathfrak{q},\sigma_\mathfrak{q})$. Furthermore we have
$$
L(s,(\pi_\mathfrak{p}\otimes\chi_\mathfrak{p})\times\sigma_\mathfrak{p})=1,
$$
because $\pi_\mathfrak{p}$ and $\sigma_\mathfrak{p}$ are unramified at $\mathfrak{p}$, but $\chi_\mathfrak{p}$ is not. At $\mathfrak{p}$ we define the Whittaker function
$$
g\mapsto
\chi_{\mathfrak{p}}(\det(g))\cdot w_{\mathfrak{p}}\left(g\cdot \tilde{j}(h^{(f)})\right)
=:w_{\mathfrak{p},\chi_\mathfrak{p}}(g),
$$
where $\tilde{j}:\GL_{n+1}\to\GL_{m}$ denotes the usual inclusion. Then corollary \ref{kor:schmidtformel} yields
$$
\Psi(w_{\mathfrak{p},\chi_\mathfrak{p}}, v_\mathfrak{p}, s)=
$$
$$
\delta(w_\mathfrak{p},v_\mathfrak{p})
\chi(-1)^{n+1}
G(\chi_\mathfrak{p})^{\frac{n(n+1)}{2}}
\absNorm(\mathfrak{f})^{-\sum_{k=1}^{n}k(n+1-k)}.
$$
Composition of these local Whittaker functions to a global Whittaker function gives the formula of the global Birch lemma in a right half plane by Fourier transform. By analytic continuation we get the formula for any $s\in\CC$, concluding the proof.
\end{proof}

An immediate consequence is
\begin{corollary}
In the case $m=n+1$ we have
$$
P(s)\cdot
\delta(w_{\mathfrak{p}},v_{\mathfrak{p}})\cdot
G(\chi)^{\frac{n(n+1)}{2}}
\absNorm(\mathfrak{f})^{-\sum_{k=1}^{n}k(n+1-k)}
L(s,(\pi\otimes\chi)\times\sigma)=
$$
$$
\sum_{\iota}
\int_{\GL_{n}(k)\backslash\GL_{n}(\Adeles_k)}
\phi_\iota
\left(
j(g)
\cdot
h^{(f)}
\right)
\varphi_\iota(g)
\chi(\det(g))
\absnorm{\det(g)}^{s-\frac{1}{2}}
dg.
$$
\end{corollary}

Let $U_\mathfrak{q}:=\Gm(\OO_{F_\mathfrak{q}})$ for nonarchimedean $\mathfrak{q}$ and define $U_\mathfrak{q}:=\Gm(k_\mathfrak{q})^0$ for $\mathfrak{q}\in S_\infty$. For an id\`ele $\alpha\in\Adeles_k^\times$ we let $C_{\alpha,\mathfrak{f}}$ denote the preimage of
$$
k^\times\backslash{}k^\times\cdot\alpha\cdot(1+\mathfrak{f})\cdot\prod_{\mathfrak{q}\nmid\mathfrak{f}} U_\mathfrak{q}
$$
under the determinant map
$$
\det: \GL_{n}(k)\backslash{}\GL_{n}(\Adeles_k)\to k^\times\backslash{}\Adeles_k^\times.
$$
Finally let $\varepsilon_x:=\diag(x,1,\dots,1)\in\GL_{n+1}(\Adeles_k)$ for $x\in \Adeles_k^\times$. As a consequence of the preceding corollary we get
\begin{corollary}\label{kor:korglobalesbirchlemma}
For any $\chi$ of finite order and conductor $\mathfrak{f}$ we have
$$
P(\frac{1}{2})\cdot 
\delta(w_\mathfrak{p}, v_\mathfrak{p})
G(\chi)^{\frac{n(n+1)}{2}}
\absNorm(\mathfrak{f})^{-\sum_{k=1}^{n}k(n+1-k)}
L(\frac{1}{2},(\pi\otimes\chi)\times\sigma)=
$$
$$
\sum_{\iota,\alpha}
\chi(\alpha)\cdot
\sum_{x}
\chi(x)\cdot
\int_{C_{\alpha,\mathfrak{f}}}
\phi_\iota
\left(
j(g)
\cdot
\varepsilon_x\cdot
h^{(f)}
\right)
\cdot
\varphi_\iota(g)
dg.
$$
Here $\alpha$ runs through a system of representatives of the class group $k^\times\backslash\Adeles_k^\times/\prod_\mathfrak{q} U_\mathfrak{q}$ and $x\in\OO_{k,\mathfrak{p}}$ runs through a system of representatives for $\left(\OO_k/\mathfrak{f}\right)^\times$.
\end{corollary}

In the function field case the sum over $\alpha$ is countably infinite but absolutely convergent.

Note that in the number field case $k^\times\backslash\Adeles_k^\times/\prod_\mathfrak{q} U_\mathfrak{q}$ not only has the classical ideal class group as a factor group, but also contains the group
$$
\left(\OO_k^\times/\OO_{k,+}^\times\right)\backslash\left(\Gm(k_\RR)/\Gm(k_\RR)^0\right)\cong
\left(\OO_k^\times/\OO_{k,+}^\times\right)\backslash\prod_{v\;\text{real}}\left(\RR^\times/(\RR^\times)^0\right),
$$
which is the kernel of the canonical map on the ideal class group. Here $\OO_{k,+}^\times=\OO_k^\times\cap\Gm(k_\RR)^0$ denotes the subroup of totally positive elements in $\OO_k^\times$.

\section{Hecke relations and distributions}

For the classical theory of Hecke operators we refer to \cite[chapter 3]{book_shimura1971} and \cite[chapter 2, \S7]{book_miyake1989}. Let $G$ be a group. A {\em Hecke pair} $(R,S)$ (in $G$) consists of a subgroup $R\leq G$ and of a sub half group $S\subseteq G$ with $RS=SR=S$ and the additional property that for any $s\in S$ the double coset $RsR$ is a finite union of right (or left) cosets modulo $R$.

The condition $RS=SR$ always holds in the case $S=G$. The second condition is satisfied if $G$ is a locally compact topological group and $R$ is a compact open subgroup.

For any Hecke pair $(R,S)$ we have a natural embedding of the free $\ZZ$-module $\mathcal H_\ZZ(R,S)$ over the set of all double cosets $RsR$ into the free $\ZZ$-module $\mathscr R_\ZZ(R,S)$ over the set of the right cosets $sR$, $s\in S$, which is induced by
$$
RsR=\bigsqcup_i s_iR\mapsto\sum_i s_iR.
$$
We may identify $\mathcal H_\ZZ(R,S)$ with its image under this embedding. Then $\mathcal H_\ZZ(R,S)$ becomes the $\ZZ$-module of $R$-invariants under the action
$$
R\times \mathscr R_\ZZ(R,S)\to \mathscr R_\ZZ(R,S),\;\;\;(r,sR)\mapsto rsR.
$$
Finally $\mathcal H_\ZZ(R,S)$ admits a structure of an associative $\ZZ$-algebra with the multiplication
$$
\left(\sum_i s_iR\right)\cdot\left(\sum_jt_jR\right):=\sum_{i,j}s_it_jR.
$$
This algebra is unitary if and only if $R\cap S\neq\emptyset$. For any commutative ring $A$ we let
$$
\mathcal H_A(R,S):=\mathcal H_\ZZ(R,S)\otimes_\ZZ A.
$$
$\mathcal H_A(R,S)$ is an associative algebra over $A$. We define $\mathcal H(R,S):=\mathcal H_\CC(R,S)$ and call it the {\em Hecke algebra} of the pair $(R,S)$.

Now let $G$ be a locally compact topological group and fix a compact open subgroup $K\leq G$. In this case $\mathscr R_\ZZ(K,G)$ may be interpreted as the $\ZZ$-module of locally constant right $K$-invariant mappings $f:G\to\ZZ$ with compact support and $\mathcal H_\ZZ(K,G)$ is just the submodule of left $K$-invariant mappings. The multiplication is given by convolution
$$
\alpha*\beta\;:\;
x\mapsto \int_G \alpha(g)\beta(xg^{-1})dg,
$$
where $dg$ is the right invariant Haar measure on $G$ which gives $K$ measure $1$. This interpretation generalizes to any Hecke algebra $\mathcal H_A(R,S)$ over any subring $A\subseteq\CC$.

All Hecke algebras we consider arise in this topological context. We have the elementary 
\begin{proposition}\label{prop:heckeinjektion}
Let $G$ denote a locally compact group, $H\leq G$ a closed subgroup and let $K\leq G$ be a compact open subgroup such that $L= H\cap K$ and $HK=G$. Then the restriction
$$
\alpha\mapsto \alpha|_H
$$
defines a monomorphism $\mathcal H_A(K,G)\to\mathcal H_A(L,H)$ of $A$-algebras.
\end{proposition}

The proof is elementary and may be found in \cite[Proposition 3.1.6]{book_andrianov1987}.

Fix a global field $k$ and a finite place $\mathfrak{p}$ of $k$. The Hecke algebra for the pair $(K,G)$ given by $K=\GL_n(\OO_{k_\mathfrak{p}})$ and $G=\GL_n(k_\mathfrak{p})$ is commonly referred to as the {\em standard Hecke algebra} at $\mathfrak{p}$. Following Tamagawa \cite{tamagawa1963} (see \cite[Theorem 6]{satake1963} as well) we have an isomorphism (the so-called {\em Satake map})
$$
\mathcal S:\mathcal H(K,G)\to \CC[X_1^{\pm1},\dots,X_n^{\pm1}]^{S_n},
$$
$$
T_\nu\mapsto \absNorm(\mathfrak{p})^{\frac{\nu(\nu+1)}{2}}\cdot\sigma_\nu(X_1,\dots,X_n),\;\;\;(0\leq\nu\leq n)
$$
where $S_n$ is the symmetric group, operation by permutation on the $X_i$, and
$$
T_\nu:=K\begin{pmatrix}{\bf1}_{n-\nu}&0\\0&\varpi\cdot {\bf1}_\nu\end{pmatrix}K
$$
is independent of the choice of a prime $\varpi$. Furthermore $\sigma_\nu$ is the elementary symmetric polynomial of degree $\nu$ in $X_1,\dots,X_n$.

As before $B_n(k_\mathfrak{p})$ denotes the standard Borel sugroup of $\GL_n(k_\mathfrak{p})$. We define $K_{B_\mathfrak{p}}:=B_n(k_\mathfrak{p})\cap K=B_n(\OO_{k_\mathfrak{p}})$ and the {\em parabolic Hecke algebra} as $\mathcal H_{B_\mathfrak{p}}:=\mathcal H(K_{B_\mathfrak{p}},B_n(k_\mathfrak{p}))$. Then Iwasawa decomposition \cite[Proposition 2.33]{iwahorimatsumoto1965}, \cite[Section 8.2]{satake1963} guarantees that the hypothesis of Proposition \ref{prop:heckeinjektion} is fulfilled and we see that $\mathcal H_{B_\mathfrak{p}}$ is a ring extension of $\mathcal H_\mathfrak{p}:=\mathcal H(K,G)$, with respect to the explicit embedding $\epsilon:\mathcal H_\mathfrak{p}\to\mathcal H_{B_\mathfrak{p}}$ given by
$$
\sum_i a_i\cdot g_iK\mapsto \sum_i a_i\cdot g_iK_{B_\mathfrak{p}},
$$
where we may assume that $g_i\in B_n(k_\mathfrak{p})$ thanks to the Iwasawa decomposition.

In $\mathcal H_{B_\mathfrak{p}}$ we have the Hecke operators
$$
U_i:=K_{B_\mathfrak{p}}
\begin{pmatrix}
{\bf1}_{i-1}&0&0\\
0&\varpi&0\\
0&0&{\bf1}_{n-i}
\end{pmatrix}
K_{B_\mathfrak{p}},
$$
which commute \cite[Lemma 2]{gritsenko1992}. Gritsenko \cite[Theorem 2]{gritsenko1992} showed that over $\mathcal H_{B_\mathfrak{p}}$ we have a decomposition of the Hecke polynomial
$$
H_\mathfrak{p}(X):=\sum_{\nu=0}^n (-1)^\nu \absNorm(\mathfrak{p})^{\frac{(\nu-1)\nu}{2}}T_\nu X^{n-\nu}\in\mathcal H_\mathfrak{p}(X)
$$
into linear factors
$$
H_\mathfrak{p}(X)=\prod_{i=1}^n(X-U_i).
$$
Following \cite[section 4]{kazhdanmazurschmidt2000} we define for $1\leq \nu\leq n$ the operators
$$
V_{\mathfrak{p},\nu}:=\absNorm(\mathfrak{p})^{-\frac{(\nu-1)\nu}{2}}\cdot U_1 U_2\cdots U_\nu\in\mathcal H_{B_\mathfrak{p}},
$$
and
$$
t_{(\mathfrak{p})}:=\diag(\varpi^{n-1},\varpi^{n-2},\dots,1).
$$
In complete analogy with \cite[Lemma 4.1]{kazhdanmazurschmidt2000} we then have
\begin{lemma}\label{lem:hecke1}
We have
$$
V_{\mathfrak{p},\nu}=
K_{B_\mathfrak{p}}
\begin{pmatrix}
\varpi\cdot {\bf1}_{\nu}&0\\
0& {\bf1}_{n-\nu}
\end{pmatrix}
K_{B_\mathfrak{p}}=
\bigsqcup_A
\begin{pmatrix}
\varpi\cdot {\bf1}_{\nu}& A\\
0& {\bf1}_{n-\nu}
\end{pmatrix}
K_{B_\mathfrak{p}},
$$
where $A\in \OO_{k_\mathfrak{p}}^{\nu\times n-\nu}$ runs through is a systemp of representatives modulo $\mathfrak{p}$. Furthermore the Hecke operators $V_{\mathfrak{p},\nu}$ commute and
$$
K_{B_\mathfrak{p}}t_{(\mathfrak{p})}K_{B_\mathfrak{p}}=
\prod_{\nu=1}^{n-1}V_{\mathfrak{p},\nu}=
\bigsqcup_u
ut_{(\mathfrak{p})}K_{B_\mathfrak{p}},
$$
where $u$ runs through a system of representatives of $U_n(\OO_{k_\mathfrak{p}})/t_{(\mathfrak{p})}U_n(\OO_{k_\mathfrak{p}})t_{(\mathfrak{p})}^{-1}$.
\end{lemma}

Denote by $\mathcal M_\mathfrak{p}$ the $\CC$-vektor space of $\CC$-valued right $K_{B_\mathfrak{p}}$-invariant mappings on $\GL_n(k_\mathfrak{p})$. The Hecke algebra $\mathcal H_{B_\mathfrak{p}}$ operates from the left on $\mathcal M_\mathfrak{p}$ by the rule
$$
\mathcal H_{B_\mathfrak{p}}\times\mathcal M_\mathfrak{p}\to\mathcal M_\mathfrak{p}
$$
$$
\left(\sum_i a_i\cdot g_iK_{B_\mathfrak{p}},\psi\right)\;\mapsto\;
\sum_i a_i\cdot [g\mapsto \psi(gg_i)].
$$

We let
$$
V_{\mathfrak{p},0}:=K_{B_\mathfrak{p}}{\bf1}_nK_{B_\mathfrak{p}}
$$
denote the unit element of $\mathcal H_{B_\mathfrak{p}}$. We have as in \cite[Proposition 4.2]{kazhdanmazurschmidt2000}
\begin{proposition}\label{prop:heckemodifikation}
Let $\underline\lambda=(\lambda_1,\dots,\lambda_{n-1})\in\CC^{n-1}$ and $\psi\in\mathcal M_\mathfrak{p}$ such that
$$
\forall \nu=1,2,\dots,n-1:\;\;\;H_\mathfrak{p}(\lambda_\nu)\cdot \psi=0.
$$
Then
$$
\psi_{\underline\lambda}:=
\prod_{i=1}^{n-1}
\prod_{\begin{subarray}cj=1\\j\neq i\end{subarray}}^{n}
(\lambda_i\absNorm(\mathfrak{p})^{1-j}V_{\mathfrak{p},j-1}-V_{\mathfrak{p},j})\cdot \psi
$$
is a simultaneous eigenfunction of $V_{\mathfrak{p},1},\dots,V_{\mathfrak{p},n-1}$. More precisely with
$$
\eta_\nu:=\absNorm(\mathfrak{p})^{-\frac{\nu(\nu-1)}{2}}\prod_{i=1}^\nu\lambda_i
$$
for $1\leq\nu\leq n-1$ we have the relation
$$
V_{\mathfrak{p},\nu}\cdot\psi_{\underline\lambda}=\eta_\nu\cdot\psi_{\underline\lambda}.
$$
\end{proposition}

Now let $\pi$ and $\sigma$ denote automorphic representations of $\GL_n$ and $\GL_{n-1}$ respectively, unramified at $\mathfrak{p}$. Let
$$
\lambda_{\mathfrak{p},1},\dots,\lambda_{\mathfrak{p},n}\in\overline{\QQ}
$$
and
$$
\alpha_{\mathfrak{p},1},\dots,\alpha_{\mathfrak{p},n-1}\in\overline{\QQ}
$$
denote the roots of the corresponding Hecke polynomials $H_\mathfrak{p}$ of $\pi_\mathfrak{p}$ and $\sigma_\mathfrak{p}$ respectively. If $\pi$ and $\sigma$ are cohomological, then $\pi_\mathfrak{p}$ and $\sigma_\mathfrak{p}$ are definied over a number field \cite[Th\'eor\`eme 3.13 resp. Proposition 3.16]{clozel1990} and consequently the Hecke roots are algebraic in this case. We say that $\pi$ (resp. $\sigma$) are {\em ordinary} at $\mathfrak{p}$, if (with a suitable numbering) for $1\leq i\leq n-1$
$$
\absnorm{\lambda_{\mathfrak{p},i}}_\mathfrak{p}=\absnorm{\absNorm(\mathfrak{p})}_\mathfrak{p}^{i-1}
$$
(resp. for $1\leq j\leq n-2$ $\absnorm{\alpha_{\mathfrak{p},j}}_\mathfrak{p}=\absnorm{\absNorm(\mathfrak{p})}_\mathfrak{p}^{j-1}$).
We write
$$
\underline{\lambda}(\mathfrak{p})=:
(\lambda_{\mathfrak{p},1},\dots,\lambda_{\mathfrak{p},n-1})\in\overline{\QQ}^{n-1},
$$
$$
\underline{\alpha}(\mathfrak{p})=:
(\alpha_{\mathfrak{p},1},\dots,\alpha_{\mathfrak{p},n-2})\in\overline{\QQ}^{n-2},
$$
and furthermore
$$
\kappa_{\underline{\lambda}(\mathfrak{p})}:=
\prod_{\nu=1}^{n-1}
\lambda_{\mathfrak{p},\nu}^{n-\nu},
$$
$$
\kappa_{\underline{\alpha}(\mathfrak{p})}:=
\prod_{\nu=1}^{n-2}
\alpha_{\mathfrak{p},\nu}^{n-1-\nu},
$$
$$
\hat\kappa_{\underline{\lambda}(\mathfrak{p})}:=
\absNorm(\mathfrak{p})^{-\frac{n(n-1)(n-2)}{6}}\cdot
\kappa_{\underline{\lambda}(\mathfrak{p})},
$$
$$
\hat\kappa_{\underline{\alpha}(\mathfrak{p})}:=
\absNorm(\mathfrak{p})^{-\frac{(n-1)(n-2)(n-3)}{6}}\cdot
\kappa_{\underline{\alpha}(\mathfrak{p})}.
$$
Under the ordinarity assumption $\hat\kappa_{\underline{\lambda}(\mathfrak{p})}$ and $\hat\kappa_{\underline{\alpha}(\mathfrak{p})}$ are $\mathfrak{p}$-adic units.

We may assume that the $\mathfrak{p}$-factor of the Fourier transform of the automorphic forms $\phi_\iota$ resp. $\varphi_\iota$ is class-1. Then $\phi_\iota$ and $\varphi_\iota$ are normalized eigenvectors of the corresponding Hecke algebras $\mathcal{H}_{\mathfrak{p}}$ at $\mathfrak{p}$. By Proposition \ref{prop:heckemodifikation} we get modified $\mathfrak{p}$-Iwahori invariant automorphic forms $\tilde{\phi}_\iota$ resp. $\tilde{\varphi}_\iota$ which are eigenvectors of the corresponding operators
$$
V_{\mathfrak{p}}:=V_{\mathfrak{p},1}\cdots V_{\mathfrak{p},n-1}
$$
with eigenvalues $\hat\kappa_{\underline{\lambda}(\mathfrak{p})}$ resp. $\hat\kappa_{\underline{\alpha}(\mathfrak{p})}$.

For any nontrivial $\mathfrak{p}$-power $\mathfrak{f}$ let
$$
\kappa(\mathfrak{f}):=
\frac{
\absNorm(\mathfrak{f})^{\frac{(n+1)n(n-1)+n(n-1)(n-2)}{6}}}
{\left(
\hat{\kappa}_{\underline{\lambda}(\mathfrak{p})}\cdot
\hat{\kappa}_{\underline{\alpha}(\mathfrak{p})}
\right)^{\nu_\mathfrak{p}(\mathfrak{f})}
},
$$
and for $\alpha\in\Adeles_k^\times$ and $x\in\OO_k^\times$
$$
\mu_\alpha(x+\mathfrak{f}):=
\kappa(\mathfrak{f})\cdot\sum_{\iota}P_{\alpha,\iota}(\varepsilon_x\cdot h^{(f)},\mathfrak{f}),
$$
where
$$
P_{\alpha,\iota}(u,\mathfrak{f}):=
\int_{C_{\alpha,\mathfrak{f}}}
\tilde{\phi}_\iota\left(
j(g)\cdot
u
\right)\cdot
\tilde{\varphi}_\iota(g)dg.
$$
Here $h^{(f)}$ is an element of $\GL_n(k_\mathfrak{p})$. Furthermore let
$$
\Theta:=
k^\times\backslash\Adeles_k^\times/\prod_{\mathfrak{q}\nmid\mathfrak{p}}U_\mathfrak{q}
\cong
\varprojlim_{\mathfrak{f}}k^\times\backslash\Adeles_k^\times/(1+\mathfrak{f})\prod_{\mathfrak{q}\nmid \mathfrak{f}}U_\mathfrak{q},
$$
and
$$
\Theta(\alpha):=
k^\times\backslash k^\times\cdot\alpha\cdot\prod_{\mathfrak{q}}U_\mathfrak{q}/\prod_{\mathfrak{q}\nmid\mathfrak{p}}U_\mathfrak{q}
\cong
\varprojlim_{\mathfrak{f}}
k^\times\backslash k^\times\cdot\alpha\cdot\prod_{\mathfrak{q}}U_\mathfrak{q}/(1+\mathfrak{f})\prod_{\mathfrak{q}\nmid\mathfrak{p}}U_\mathfrak{q}.
$$
Then $\Theta$ is a disjoint union of the compact open sets $\Theta(\alpha_1),\dots,\Theta(\alpha_h)$. We may assume that $\alpha_1,\dots,\alpha_h$ are trivial at $\mathfrak{p}$.
\begin{theorem}\label{satz:distribution}
If $\mathfrak{p}^{\frac{n(n-1)}{2}}$ is principal, then $\mu_{\alpha_1},\dots,\mu_{\alpha_h}$ are distributions on $\Theta(\alpha_i)$ which give rise to a $\CC$-valued distribution $\mu$ on $\Theta$.

For any character $\chi:k^\times\backslash\Adeles_k^\times\to\CC^\times$ of finite order with nontrivial $\mathfrak{p}$-power conductor $\mathfrak{f}$, trivial at $\infty$, we have
$$
\int_\Theta
\chi d\mu\;=\;
P(\frac{1}{2})\cdot
\delta(\pi,\sigma)\cdot
\hat{\kappa}(\mathfrak{f})\cdot
G(\chi)^{\frac{n(n-1)}{2}}\cdot
L(\frac{1}{2},(\pi\otimes\chi)\times\sigma).
$$
Here $\hat{\kappa}(\mathfrak{f})$ and $\delta(\pi,\sigma)$ are given explicitly by
$$
\hat{\kappa}(\mathfrak{f}):=
\absNorm(\mathfrak{f})^{\frac{n(n-1)(n-2)}{6}}
\cdot
(\hat{\kappa}_{\underline{\lambda}(\mathfrak{p})}\hat{\kappa}_{\underline{\alpha}(\mathfrak{p})})^{-\nu_{\mathfrak{p}}(\mathfrak{f})}
,
$$
and
$$
\delta(\pi,\sigma):=
\tilde{w}_{\mathfrak{p}}({\bf1}_n)
\cdot
\tilde{v}_{\mathfrak{p}}({\bf1}_{n-1})
\cdot
\prod_{\nu=1}^{n-1}\left(1-\absNorm(\mathfrak{p})^{-\nu}\right)^{-1}.
$$
$\tilde{w}_{\mathfrak{p}}$ and $\tilde{v}_{\mathfrak{p}}$ denote the local Whittaker functions at $\mathfrak{p}$, corresponding to the $\mathfrak{p}$-factor of the Fourier transform of $\tilde{\phi}_\iota$ and $\tilde{\varphi}_\iota$.
\end{theorem}

We could renormalize our Whittaker functions in such a way that $\delta(\pi,\sigma)=1$. The precise value of $\delta(\pi,\sigma)$ is given in \cite[Proposition 4.12]{kazhdanmazurschmidt2000} and lies in the field generated by the Hecke roots.

\begin{proof}
Our proof follows the proofs of \cite[Proposition 4.9]{kazhdanmazurschmidt2000} and \cite[Theorem 3.1]{schmidt2001}. Introduce the notation
$$
U_n:=U_n(\OO_{k_\mathfrak{p}}),
$$
and
$$
U_n^{(\varpi)}:=t_{(\mathfrak{p})} U_n t_{(\mathfrak{p})}^{-1}.
$$
Proposition \ref{prop:heckemodifikation} and Lemma \ref{lem:hecke1} give
$$
\forall g\in\GL_n(\Adeles_k):\;\;\;
\sum_{uU_n^{(\varpi)}\in U_n/U_n^{(\varpi)}}
\tilde{\phi}_\iota(gut_{(\mathfrak{p})})=
\hat\kappa_{\underline{\lambda}(\mathfrak{p})}\cdot
\tilde{\phi}_\iota(g)
$$
and
$$
\forall g\in\GL_{n-1}(\Adeles_k):\;\;\;
\sum_{uU_{n-1}^{(\varpi)}\in U_{n-1}/U_{n-1}^{(\varpi)}}
\tilde{\varphi}_\iota(gut_{(\mathfrak{p})})=
\hat\kappa_{\underline{\alpha}(\mathfrak{p})}\cdot
\tilde{\varphi}_\iota(g).
$$
We conclude that
$$
\hat\kappa_{\underline{\lambda}(\mathfrak{p})}\cdot
\hat\kappa_{\underline{\alpha}(\mathfrak{p})}\cdot
P_{\alpha,\iota}(\varepsilon_x\cdot h^{(f)},\mathfrak{f})=
$$
$$
\sum_{uU_n^{(\varpi)}\in U_n/U_n^{(\varpi)}}
\sum_{wU_{n-1}\in U_{n-1}/U_{n-1}^{(\varpi)}}
\int_{C_{\alpha,\mathfrak{f}}}
\tilde{\phi}_\iota\left(
j(g)\cdot
\varepsilon_x\cdot
h^{(f)}
ut_{(\mathfrak{p})}
\right)\cdot
\tilde{\varphi}_\iota(gwt_{(\mathfrak{p})})dg=
$$
$$
\sum_{u}
\sum_{w}
\int_{C_{\alpha,\mathfrak{f}}}
\tilde{\phi}_\iota\left(
j(g)\cdot
\varepsilon_x\cdot
t_{(\mathfrak{p})}^{-1}j(w)^{-1}
h^{(f)}
ut_{(\mathfrak{p})}
\right)\cdot
\tilde{\varphi}_\iota(g)dg,
$$
since $\det(t_{\mathfrak{p}})$ globally represents a principal ideal. This shows that
$$
P_{\alpha,\iota}(\varepsilon_x\cdot h^{(f)},\mathfrak{f})=
$$
\begin{equation}
\hat\kappa_{\underline{\lambda}(\mathfrak{p})}^{-1}\cdot
\hat\kappa_{\underline{\alpha}(\mathfrak{p})}^{-1}\cdot
\sum_{u}
\sum_{w}
P_{\alpha,\iota}(
\varepsilon_x\cdot
t_{(\mathfrak{p})}^{-1}t^{-1}tj(w)^{-1}
h^{(f)}
ut^{-1}tt_{(\mathfrak{p})}
,\mathfrak{f}).
\label{eq:satzbew}
\end{equation}
Now \cite[Lemma 3.2]{schmidt2001} easily generalizes to any local field, which means that we have
$$
tj(w)^{-1}
h^{(f)}
ut^{-1}
\;=\;
tj(w)^{-1}t^{-1}
h^{(1)}
tut^{-1}
\;\equiv\;
h^{(1)}
\;(\mmod\mathfrak{f}),
$$
since
$$
tj(w)^{-1}t^{-1}\equiv tut^{-1}\equiv{\bf1}_n\;(\mmod\mathfrak{f}).
$$
Finally we have the formula
$$
\left(U_{n}:U_{n}^{(t_\mathfrak{p})}\right)=
\absNorm(\mathfrak{p})^{\frac{(n+1)n(n-1)}{6}}
$$
cf. \cite[p. 110]{kazhdanmazurschmidt2000}, hence
$$
P_{\alpha,\iota}(\varepsilon_xh^{(f)},\mathfrak{f})=
\hat\kappa_{\underline{\lambda}(\mathfrak{p})}^{-1}\cdot
\hat\kappa_{\underline{\alpha}(\mathfrak{p})}^{-1}\cdot
\absNorm(\mathfrak{p})^{\frac{(n+1)n(n-1)+n(n-1)(n-2)}{6}}\cdot
P_{\alpha,\iota}(\varepsilon_x h^{(f\varpi)},\mathfrak{f}).
$$
Due to the relation
$$
P_{\alpha,\iota}(\varepsilon_x h^{(f\varpi)},\mathfrak{f})=
\sum_{a\;(\mmod\mathfrak{p})}
P_{\alpha,\iota}(\varepsilon_{x+af} h^{(f\varpi)},\mathfrak{fp})
$$
we conclude that
$$
\mu_{\alpha}(x+\mathfrak{f})=\sum_{a\;(\mmod\mathfrak{p})}\mu_{\alpha}(x+af+\mathfrak{fp}),
$$
proving the distribution relation. The interpolation formula follows from corollary \ref{kor:korglobalesbirchlemma}. This proves the theorem.
\end{proof}

\begin{remark}
The condition on $\mathfrak{p}$ may be weakened by a modified construction, if we forget the finer structure provided by the $\mu_{\alpha_i}$ and consider
$$
\mu':=\sum_{i=1}^h\mu_{\alpha_i}.
$$
This sum is invariant under translations, which gives the distribution relation in the general case along the same lines. But at the same time this limits integration to characters constant on $\alpha_1,\dots,\alpha_h$. If for any character $\chi$ under consideration we may find an unramified character $\chi_{\rm nr}$, such that $\chi_{\rm nr}\chi$ becomes trivial on $\alpha_1,\dots,\alpha_h$, then the transition from $\pi$ to the twist $\pi\otimes\chi_{\rm nr}^{-1}$ gives us distributions $\mu_{\chi_{\rm nr}}'$, which give rise to a distribution $\mu$ on $\Theta$. In more algebraic terms, the class group $\mathcal{C}_k$ of $k$ is in this case a direct factor of $\Theta=\Theta'\times\mathcal{C}_k$ and the Iwasawa algebra
$$
\CC[[\Theta]]=
\varprojlim\CC\big[k^\times\backslash\Adeles_k^\times/(1+\mathfrak{f})\prod_{\mathfrak{q}\nmid\mathfrak{f}}U_\mathfrak{q}\big]
$$
of $\CC$-valued distributions on $\Theta$ decomposes canonically into a tensor product $\CC[\mathcal{C}_k]\otimes_\CC\CC[[\Theta']]$ (cf. \cite[section (7.3)]{mazurswinnertondyer1974}). To be more precise, the linear independence of the unramified characters shows that there are distributions $\mu_{\chi_{\rm nr}}$ on $\mathcal{C}_k$ such that
$$
\int_{\mathcal{C}_k}\chi_i^{-1} d\mu_{\chi_j}=\delta_{ij}\;\;\;\text{(Kronecker delta)}.
$$
The distribution
$$
\mu:=\sum_{i=1}^h \mu_{\chi_i}\otimes\mu_{\chi_i}'\in\CC[[\Theta]]
$$
then has the interpolation property of Theorem \ref{satz:distribution} for {\em all} id\`ele class characters with $\mathfrak{p}$-power conductor.
\end{remark}

\begin{remark}
Note that the short exact sequence
\begin{equation}
1\to\overline{\OO_{k,+}^\times}\backslash
\OO_{\mathfrak{p}}^\times\times \pi_0(k_\RR^\times)\to
\Theta\to
\mathcal{C}_k\to 1
\label{eq:clsequenz}
\end{equation}
does not split in general. A simple counter example is given by $k=\QQ(\sqrt{-15})$. Here we have $\mathcal{C}_k\cong\ZZ/2\ZZ$ and the ray class group $\mathcal{C}^{\mathfrak{p}}$ is cyclic of order $4$ if $\mathfrak{p}\mid 5$ and cyclic of order $16$ if $\mathfrak{p}\mid 17$. In the latter case we have $\varprojlim\mathcal{C}^{\mathfrak{p}^n}=\ZZ_{17}^\times$.
\end{remark}

For almost all prime places the group $\OO_\mathfrak{p}^\times$ still contains an open pro-$p$-subgroup $H$, where $p$ denotes the residue field characteristic of $\mathfrak{p}$. Consequently the condition $(p,\#\mathcal{C}_k)=1$ guarantess that any character of $H$ as a continuation on $\Theta$ which may be assumed trivial on $\alpha_1,\dots,\alpha_h$. Due to the finiteness of $\OO_{\mathfrak{p}}^\times/H$ we still get the interpolation property for a subgroup of all characters of $\mathfrak{p}$-power conductor of finite index.

\begin{theorem}\label{satz:distribution2}
If $\mathfrak{p}$ does not divide the order of the class group of $k$, there is a distribution $\mu$ on $\theta$, such that the interpolation formula of theorem \ref{satz:distribution} holds for all nontrivial characters $\chi$ in a subgroup of finite index of all finite order characters with $\mathfrak{p}$-power conductor.
\end{theorem}

\begin{remark}
To illustrate an extreme situation in the case $p\mid\#\mathcal{C}_k$ let $k=\QQ(\sqrt{-23})$. This imaginary quadratic field has the class group $\mathcal{C}_k\cong\ZZ/3\ZZ$ and for the corresponding ray class groups are given by $\mathcal{C}^{\mathfrak{p}^r}\cong\ZZ/3^r\ZZ$ for any $r\geq 1$ if $\mathfrak{p}\mid 3$. Hence we have $\varprojlim \mathcal{C}^{\mathfrak{p}^r}=\ZZ_3$, which means that $\Theta/\pi_0(k_\RR^\times)$ is pro-cyclic. Thus it may eventually happen that {\em no} nontrivial character of $\OO_{\mathfrak{p}}^\times$ has a lift to $\Theta$ trivial on $\alpha_1,\dots,\alpha_h$.
\end{remark}

Note that Ash and Ginzburg overlooked this phenomenon in \cite{ashginzburg1994}. In section 2.2 of loc. cit. they implicitly assume that any character of $\OO_{k_\mathfrak{p}}^\times$ lifts to an id\`ele class group character trivial on $\alpha_1,\dots,\alpha_h$.

\begin{remark}
All known constructions suffer this restriction. Already Manin, constructing $\mathfrak{p}$-adic $L$-functions for Hilbert modular forms in \cite{manin1976} by basically the same method, noted that the nontriviality of the class group introduces some new phenomena. Manin argues that the effect of restricting an id\`ele class group character $\chi$ on $\OO_{k_\mathfrak{p}}^\times$ gives rise to a controllable modification in the corresponding twisted $L$-function. Consequently he would consider the distribution $\mu$ that we construct ($\mu$ always exists) as the $\mathfrak{p}$-adic analogue of the complex $L$-function $L(s,\pi\times\sigma)$.

Another analogy between our construction and Manins construction is the following. Instead of restricting to $\mathfrak{p}$-power conductors Manin considers general conductors $\mathfrak{f}$ with prime divisors in a fixed finite set $S$. Our general global Birch lemma easily generalizes to this more general situation and analoguously gives rise to a distribution $\mu$ interpolating twists with characters that are $S$-complete in the sense of Manin. This means that all places in $S$ are divisors of the conductors $\mathfrak{f}$. We restricted to the case $S=\{\mathfrak{p}\}$ to give a less technical treatment.
\end{remark}

\section{Algebraicity and boundedness of the distribution}

Fix a number field $k/\QQ$ with $r_1$ real and $r_2$ complex places. We write $S_\infty$ for its set of archimedean places. Let $G_n:=\res_{k/\QQ}\GL_n$, where by abuse of notation $\det: G_n\to G_1$ is the restriction of scalars of $\det:\GL_n\to\GL_1$. Then $\zentrum(G_n)^0=\res_{k/\QQ}\zentrum(\GL_n)^0$ and (cf. \cite[section 1.4]{ono1961})
$$
\rang_\QQ(\zentrum(G_n))=1.
$$
On the other hand
$$
\rang_{\overline{\QQ}}(\zentrum(G_n))=[k:\QQ],
$$
and it is easily seen that
$$
\rang_\RR\zentrum(G_n)=r_1+r_2.
$$
We may assume that $\GL_n(\OO)=G_n(\ZZ)$, $\GL_n(k)=G_n(\QQ)$, $\GL_n(k_\RR)=G_n(\RR)$, and $\GL_n(\Adeles_k)=G_n(\Adeles_\QQ)$.

Now let $\pi,\sigma$ be irreducible cohomological cuspidal representations of $\GL_n(\Adeles_k)$ and $\GL_{n-1}(\Adeles_k)$ with trivial central character. Choose compact open subgroups $K\leq\GL_n(\Adeles_k)$ and $K'\leq\GL_{n-1}(\Adeles_k)$ such that
$$
\det(K)=\det(K')=\widehat{\OO}_k^\times.
$$
Therefore we may assume that there is a $K$-(resp. $K'$-)right invariant new vector $w_{\rm f}$ (resp. $v_{\rm f}$) of the finite component $\pi_{\rm f}$ (resp. $\sigma_{\rm f}$). Furthermore we may assume (cf. \cite[section (4.1), Th\'eor\`eme]{jpss1981c}) that $K$ contains the image of $K'$ under the embedding
$$
j:\GL_{n-1}\to\GL_n,\;
g\mapsto
\begin{pmatrix}
g&\\
 &1
\end{pmatrix}.
$$
Finally we assume that the modified autmorphic forms $\tilde{\phi}_\iota$ and $\tilde{\varphi}_\iota$ are right-$K$- and right-$K'$-invariant respectively and that locally $K_\mathfrak{p}=I_n(\OO_{k_\mathfrak{p}})$ resp. $K_\mathfrak{p}'=I_{n-1}(\OO_{k_\mathfrak{p}})$. Let $K$ and $K'$ be small enough such that all arithmetic subgroups in the sequel are torsion free.

For any $x\in\OO_{k_\mathfrak{p}}^\times$, an id\`ele $\alpha$ and a generator $f\in\OO_{k_\mathfrak{p}}$ of a nontrivial $\mathfrak{p}$-power $\mathfrak{f}$ our aim is to give a cohomological interpretation of the integral $P_{\alpha,\iota}(\varepsilon_x h^{(f)},\mathfrak{f})$. Thanks to the $\pi_0(\GL_n(k_\RR))\times\GL_n(\Adeles_k^{\rm f})$- (resp. $\pi_0(\GL_{n-1}(k_\RR))\times\GL_{n-1}(\Adeles_k^{\rm f})$-)action on the space of automorphic forms we get
$$
P_{\alpha,\iota}(\varepsilon_x h^{(f)},\mathfrak{f})\;=\;
\int_{C_{1,\mathfrak{f}}}
\tilde{\phi}_\iota^\alpha\left(
j(g)
\varepsilon_x h^{(f)}\right)\cdot
\tilde{\varphi}_\iota^\alpha(g)dg,
$$
where $\tilde{\phi}_\iota^\alpha$ (resp. $\tilde{\varphi}_\iota$) denotes the image of $\tilde{\phi}_\iota$ under the action of $\varepsilon_\alpha$ (resp. $\varphi_\iota$) as an element of $\pi_0(G_n(\RR))\times\GL_n(\Adeles_k^{\rm f})$. We write
$$
K^\alpha:=\varepsilon_\alpha K\varepsilon_\alpha^{-1}
$$
and we have the corresponding arithmetic subgroup
$$
\Gamma_\alpha:=\{\gamma\in\GL_{n}^+(k)\mid \gamma_{\rm f}\in K^\alpha\},
$$
with $\GL_{n}^+(k):=\GL_{n}(k)\cap\GL_{n}(k_\RR)^0$. Strong approximation for $\SL_n$ yields the decomposition
$$
\GL_n(\Adeles_k)
=
\bigsqcup_i \GL_n(k)\cdot
\varepsilon_{\alpha_i}\cdot
\left(\GL_n(k_\RR)^0\times K^\alpha\right)
$$
corresponding to the fibers of the determinant map. We find $\gamma_{x,\mathfrak{f}}\in\GL_n(k)$, $\gamma_{x,\mathfrak{f},\infty}\in \GL_n(k_\RR)^0$, $g_{\rm f}\in K^\alpha$ and $1=\alpha(\mathfrak{f})\in\{\alpha_1,\dots,\alpha_h\}$ with
$$
\varepsilon_{x}\cdot h^{(f)}=\gamma_{x,\mathfrak{f}}^{-1}\cdot(\gamma_{x,\mathfrak{f},\infty},\varepsilon_{\alpha(\mathfrak{f})}\cdot g_{\rm f})=\gamma_{x,\mathfrak{f}}^{-1}\cdot(\gamma_{x,\mathfrak{f},\infty},g_{\rm f}).
$$
We define the group
$$
K_{\alpha,x,\mathfrak{f}}':=
j^{-1}\left(
j\left((K')^\alpha\right)\cap
\varepsilon_{x}
h^{(f)}
K^\alpha
{h^{(f)}}^{-1}
\varepsilon_{x}^{-1}
\right).
$$
Then $K_{\alpha,x,\mathfrak{f}}'$ operates on $C_{\alpha x,\mathfrak{f}}$ via right translation (cf. \cite[Prop. 3.4]{schmidt2001}) and we have
$$
\Gamma_\alpha':=\{\gamma\in\GL_{n-1}^+(k)\mid \gamma_{\rm f}\in (K')^\alpha\},
$$
$$
\Gamma_{\alpha,x,\mathfrak{f}}':=
\{\gamma\in\GL_{n-1}^+(k)\mid \gamma_{\rm f}\in K_{\alpha,x,\mathfrak{f}}'\}.
$$
Then
$$
\Gamma_{\alpha,x,\mathfrak{f}}'=
\{\gamma\in\Gamma_\alpha'\mid j(\gamma)_{\rm f}\in
\varepsilon_{x}
h^{(f)}
K^\alpha
{h^{(f)}}^{-1}
\varepsilon_{x}^{-1}
\}\subseteq\Gamma_{\alpha}',
$$
and we get the diffeomorphism
$$
i_{\alpha,x,\mathfrak{f}}:
\Gamma_{\alpha,x,\mathfrak{f}}'\backslash \GL_{n-1}(k_\RR)^0
\to C_{\alpha x,\mathfrak{f}}/K_{\alpha,x,\mathfrak{f}}',
$$
$$
\Gamma_{\alpha,x,\mathfrak{f}}'\cdot
g_\infty
\mapsto
\GL_{n-1}(k)\cdot (g_\infty,\varepsilon_{x})\cdot K_{\alpha,x,\mathfrak{f}}'.
$$
Since
$$
\GL_n(k)\cdot(g_\infty,\varepsilon_x\cdot h^{(f)})\cdot K^\alpha
=
\GL_n(k)\cdot
(
\gamma_{x,\mathfrak{f},
\infty}g_\infty,
1_{\rm f}
)\cdot K^\alpha
$$
we conclude that
$$
\vol\left(K_{\alpha,x,\mathfrak{f}}'\right)^{-1}\cdot
P_{\alpha,\iota}(\varepsilon_x h^{(f)},\mathfrak{f})\;=\;
$$
$$
\int_{\Gamma_{\alpha,x,\mathfrak{f}}'\backslash G_{n-1}(\RR)^0}
\tilde{\phi}_\iota^{\alpha}
\left(
\gamma_{x,\mathfrak{f},\infty}
j(g_\infty)
\right)\cdot
\tilde{\varphi}_\iota^\alpha(g_\infty)dg_\infty.
$$
Note that
$$
\Gamma_{\alpha,x,\mathfrak{f}}'=
\{\gamma\in\Gamma_\alpha'\mid j(\gamma)_{\rm f}\in
\gamma_{x,\mathfrak{f},f}^{-1}
K^{\alpha}
\gamma_{x,\mathfrak{f},f}
\}=:
\Gamma_{\alpha,\gamma_{x,\mathfrak{f}}}'.
$$

We give a direct cohomological interpretation of this integral. Let $\mathscr X_n$ be the symmetric space of $G_n(\RR)$ with respect to the standard maximal compact subgroup $K_\infty$, hence canonically
$$
\mathscr X_n=\prod_{\mathfrak{q}\in S_\infty\;\text{real}}\GL_n(\RR)/\Oo(n)\times \prod_{\mathfrak{q}\in S_\infty\;\text{complex}}\GL_n(\CC)/\U(n).
$$
Let $\mathscr X_n^1$ be the symmetric space to the standard maximal compact subgroup of $G_n^{\ad}(\RR)$. By means of the canonical isogeny $G_n^{\der}\to G_n^{\ad}$ we may consider $\mathscr X_n^1$ as a symmetric space of $G_n^{\der}(\RR)$ as well.

We write
$$
b_n:=\frac{n^2-n+2\left[\frac{n}{2}\right]}{4},
$$
$$
\tilde{b}_n:=\frac{n(n-1)}{2},
$$
$$
c_n:=\dim(\liegl_n)-\dim(\lieo_n)=\frac{n^2+n}{2},
$$
$$
\tilde{c}_n:=\dim_\RR(\liegl_n\otimes_\RR\CC)-\dim_\RR(\U(n))=2n^2-n^2=n^2,
$$
Then $b_n+b_{n-1}=c_{n-1}$, $\tilde{b}_n+\tilde{b}_{n-1}=\tilde{c}_{n-1}$, and
$$
r_1c_n+r_2\tilde{c}_n=\dim\mathscr X_n=\dim\mathscr X_n^1+r_1+r_2.
$$
Denote by $\tilde{\lieg}_n$ the Lie subalgebra of
$$
\lieg_n:=\Lie(G_n(\RR))=\bigoplus_{\mathfrak{q}\in S_\infty\;\text{real}}\liegl_n\oplus\bigoplus_{\mathfrak{q}\in S_\infty\;\text{complex}}\liegl_n\otimes_\RR\CC
$$
given by matrices with componentwise totally imaginary trace. We write $\tilde{\liesl}_n$ for the Lie subalgebra of $\liegl_n\otimes_\RR\CC$ given by matrices with totally imaginary trace.

We may assume that $\pi$ and $\sigma$ occur in dimension $(r_1b_n+r_2\tilde{b}_n)$ resp. $(r_1b_{n-1}+r_2\tilde{b}_{n-1})$ of the cohomology of the corresponding Lie algebras. Then $\pi$ and $\sigma$ occur with multiplicity one. More precisely
$$
H^{r_1b_n+r_2\tilde{b}_n}\left(\tilde{\lieg}_n,K_\infty;H_{\pi_{\infty}}^{(K_\infty)}\right)\cong\CC,
$$
where
$$
\pi_\infty=\Otimes\limits_{\mathfrak{q}\mid\infty}\pi_\mathfrak{q}
$$
and $H_{\pi_{\infty}}^{(K_\infty)}$ is the space of $K_\infty$-finite elements of the representation space $H_{\pi_\infty}$ of $\pi_\infty$. We implicitly used that $\pi_\infty$ is uniquely determined by its (irreducible) restriction to
$$
G_n^\pm:=\{g\in\GL_n(k_\RR)\mid \forall \mathfrak{q}\in S_\infty: \absnorm{\det(g_\mathfrak{q})}_\mathfrak{q}=1\}
$$
The Lie algebra of this group is just $\tilde{\lieg}_n$. The claimed multiplicity one follows from \cite[Lemme 3.14]{clozel1990} via K\"unneth formalism \cite[chap. I, \S1.3 and \S5, (4)]{book_borelwallach1980}, which in our case reads
\begin{equation}
H^{r_1b_n+r_2\tilde{b}_n}\left(\tilde{\lieg},K_\infty,H_{\pi_{\infty}}^{(K_\infty)}\right)\cong
$$
$$
\bigotimes_{v\;\text{real}}
H^{b_n}\left(\liesl_n,\Oo(n),H_{\pi_{\mathfrak{q}}}^{(\Oo(n))}\right)
\otimes
\bigotimes_{v\;\text{complex}}
H^{\tilde{b}_n}\left(\tilde{\liesl_n},\U(n),H_{\pi_{\mathfrak{q}}}^{(\U(n))}\right),
\label{eq:kuenneth}
\end{equation}
since for $r<b_n$
$$
H^{r}\left(\liesl_n,\Oo(n),H_{\pi_{\mathfrak{q}}}^{(\Oo(n))}\right)=0,
$$
resp. for $s<\tilde{b}_n$
$$
H^{s}\left(\tilde{\liesl}_n,\U(n),H_{\pi_{\mathfrak{q}}}^{(\U(n))}\right)=0.
$$
We have the $(\tilde{\lieg}_n,K_\infty)$-module
$$
\mathscr W_0(\pi_\infty,\psi_\infty):=\bigotimes_{\mathfrak{q}\mid\infty}\mathscr W_0(\pi_\mathfrak{q},\psi_\mathfrak{q}).
$$
By \cite[chap. II, \S3, Corollary 3.2]{book_borelwallach1980} we have
$$
H^{r_1b_n+r_2\tilde{b}_n}\left(\tilde{\lieg}_n,K_\infty,H_{\pi_{\infty}}^{(K_\infty)}\right)\cong
\left(
\bigwedge^{r_1b_n+r_2\tilde{b}_n}\tilde{\lieh}_n^*\otimes\mathscr W_0(\pi_\infty,\psi_\infty)
\right)^{K_\infty},
$$
where $\tilde{\lieh}_n$ denotes the $(-1)$-eigenspace of the Cartan involution in $\tilde{\lieg}$.

Now let
$$
0\neq\eta_\infty\in\left(
\bigwedge^{r_1b_n+r_2\tilde{b}_n}\tilde{\lieh}_n^*\otimes\mathscr W_0(\pi_\infty,\psi_\infty)
\right)^{K_\infty}.
$$
We write
$$
\eta_\infty=\sum_{\#I=r_1b_n+r_2\tilde{b}_n}\omega_I\otimes w_{I,\infty}
$$
with $w_{I,\infty}\in\mathscr W_0(\pi_\infty,\psi_\infty)$, similarly for $\sigma$. We get a form
$$
\eta_\infty'=\sum_{\#I'=r_1b_{n-1}+r_2\tilde{b}_{n-1}}\omega_{I'}'\otimes v_{I',\infty}
$$
of degree $r_1b_{n-1}+r_2\tilde{b}_{n-1}$ with $v_{I',\infty}\in\mathscr W_0(\sigma_\infty,\overline{\psi}_\infty)$. For the pair $(w_{\iota,{\rm f}},v_{\iota,{\rm f}})$ of finite Whittaker functions corresponding to our forms $(\tilde{\phi}_\iota,\tilde{\varphi}_\iota)$ this gives
$$
\eta_{\iota,0}:=w_{\iota,{\rm f}}\otimes\eta_\infty\in
\left(
\bigwedge^{r_1b_n+r_2\tilde{b}_n}\tilde{\lieh}_n^*\otimes\mathscr W_0(\pi,\psi)
\right)^{K_\infty},
$$
and
$$
\eta_{\iota,0}':=v_{\iota,{\rm f}}\otimes\eta_\infty'\in
\left(
\bigwedge^{r_1b_{n-1}+r_2\tilde{b}_{n-1}}\tilde{\lieh}_{n-1}^*\otimes\mathscr W_0(\sigma,\overline{\psi})
\right)^{K_\infty'}.
$$
Coefficientwise Fourier transform yields
$$
\eta_{\iota}\in
\left(
\bigwedge^{r_1b_n+r_2\tilde{b}_n}\tilde{\lieh}_n^*\otimes L_0^2(\GL_n(k)\backslash\GL_n(\Adeles_k)/(\zentrum(\GL_n)(k_\RR)^0K_\infty K))
\right)^{K_\infty},
$$
and analoguously an $\eta_{\iota}'$. Hence we get cohomology classes
$$
[\eta_{\alpha,\iota}]\in H_{\rm cusp}^{r_1b_n+r_2\tilde{b}_n}(\Gamma_\alpha\backslash\mathscr X_n^1,\CC)\subseteq H_{\rm c}^{r_1b_n+r_2\tilde{b}_n}(\Gamma_\alpha\backslash\mathscr X_n^1,\CC),
$$
$$
[\eta_{\alpha,\iota}']\in H_{\rm cusp}^{r_1b_{n-1}+r_2\tilde{b}_{n-1}}(\Gamma_\alpha'\backslash\mathscr X_{n-1}^1,\CC)\subseteq H_{\rm c}^{r_1b_n+r_2\tilde{b}_{n-1}}(\Gamma_\alpha'\backslash\mathscr X_{n-1}^1,\CC),
$$
where the arithmetic groups operate via the adjoint action. Let $\phi_{\iota,I}$ be the Fourier transform of $w_{I,\infty}\otimes w_{\iota,{\rm f}}$, and we fix the notation
$$
\eta_\iota=\sum_{\#I=r_1b_n+r_2\tilde{b}_n}\omega_I\otimes\phi_{\iota,I}
$$
and
$$
\eta_\iota'=\sum_{\#I'=r_1b_{n-1}+r_2\tilde{b}_{n-1}}\omega_{I'}'\otimes\varphi_{\iota,I'}
$$
respectively. By Lemma \ref{lem:eigentlich} $\ad$ induces a proper projection
$$
p:{}^0\mathscr X_{n-1}\to\mathscr X_{n-1}^1.
$$
Furthermore we have the central morphism $\ad\circ j:\GL_{n-1}\to\PGL_n$ and its restriction of scalars
$$
s:=\res_{k/\QQ}\ad\circ j=\ad\circ\res_{k/\QQ}j:G_{n-1}\to G_n^{\ad}.
$$
Finally we have the cohomological formula
$$
\vol\left(K_{\alpha,x,\mathfrak{f}}'\right)^{-1}\cdot
P_{\alpha,\iota}(\varepsilon_x h^{(f)},\mathfrak{f})\;=\;
$$
$$
\int_{\Gamma_{\alpha,\gamma_{x,\mathfrak{f}}}'\backslash{}^0\mathscr X_{n-1}}
s_{\alpha,\gamma_{x,\mathfrak{f}},*}\left(\eta_{\alpha,\iota}\right)\wedge p^*(\eta_{\alpha,\iota}')=\mathscr P_{s_\alpha,\gamma_{x,\mathfrak{f}}}^{r_1b_n+r_2\tilde{b}_n}([\eta_{\alpha,\iota}],[\eta_{\alpha,\iota}']).
$$

\begin{theorem}\label{thm:algebraicity}
Let $\pi$ and $\sigma$ be ordinary at $\mathfrak{p}$. By suitable cohomological choice of the Whittaker functions at $\infty$ Theorems \ref{satz:distribution} and \ref{satz:distribution2} give distributions with values in the ring of integers $\OO_E$ of a field $E=\QQ(\pi,\sigma)$, i.e. they are $\mathfrak{p}$-adically bounded.
\end{theorem}

\begin{proof}
Due to multiplicity one we may assume that one, and hence all, classes lie in the cuspidal cohomology
$$
[\eta_{\alpha,\iota}]\in H_{\rm cusp}^{r_1b_n+r_2\tilde{b}_n}(\Gamma_\alpha\backslash\mathscr X_n^1,\OO_{\QQ(\pi,\sigma)}),
$$
$$
[\eta_{\alpha,\iota}']\in H_{\rm cusp}^{r_1b_{n-1}+r_2\tilde{b}_{n-1}}(\Gamma_\alpha'\backslash\mathscr X_{n-1}^1,\OO_{\QQ(\pi,\sigma)})
$$
with {\em entire} coefficients. This immediately implies that
$$
\vol\left(K_{\alpha,x,\mathfrak{f}}'\right)^{-1}\cdot
P_{\alpha,\iota}(\varepsilon_x h^{(f)},\mathfrak{f})\;\in\;\OO_{\QQ(\pi,\sigma)}.
$$

To see the boundedness of the distribution, we must show that
$$
\left|\kappa(\mathfrak{f})\cdot \vol\left(K_{\alpha,x,\mathfrak{f}}'\right)\right|_p\;=\;
\left|
\absNorm(\mathfrak{f})^{\frac{(n+1)n(n-1)+n(n-1)(n-2)}{6}}\cdot
\vol\left(K_{\alpha,x,\mathfrak{f}}'\right)
\right|_p
$$
is bounded for any nontrivial $\mathfrak{p}$-power $\mathfrak{f}$. Since $\alpha$ and $x$ commute, $\varepsilon_x$ normalizes the group $(K')^\alpha$. Therefore
$$
K_{\alpha,x,\mathfrak{f}}'\;=\;
\varepsilon_x\cdot K_{\alpha,1,\mathfrak{f}}'\cdot\varepsilon_x^{-1}.
$$
Conjugation with $h^{(f)}$ affects only the components at $\mathfrak{p}$, hence for $\mathfrak{q}\nmid\mathfrak{f}$ we have
$$
\left(K_{\alpha,x,\mathfrak{f}}'\right)_\mathfrak{q}\;=\;j(K_\mathfrak{q}')^\alpha.
$$
Therefore we are left to consider
$$
\left(K_{\alpha,x,\mathfrak{f}}'\right)_\mathfrak{p}\;=\;j(I_{n-1})^\alpha\cap
h^{(f)} I_{n}^\alpha {h^{(f)}}^{-1}.
$$
We may assume that $\alpha$ has a trivial $\mathfrak{p}$-component. Then our claim follows from the natural generalization of \cite[Lemmas 3.6 and 3.7]{schmidt2001}.
\end{proof}

\begin{remark}
The proof of the theorem shows that we get an estimate of the order of the distribution along the same lines if we drop the assumption of ordinarity.
\end{remark}

Now assume $k$ to be totally real. Then cohomologicality of $\pi$ means that for any $\mathfrak{q}\in S_\infty$
$$
H^\bullet(\liegl_n,\SO(n)\zentrum(\GL_n)(\RR)^0;H_{\pi_\mathfrak{q}}^{(\Oo(n))})\neq 0
$$
and we have a similar statement for $\sigma$. The question wether $P(\frac{1}{2})\neq 0$ is a local problem, because the K\"unneth formalism yields factorizations
$$
\eta_\infty=\Otimes\limits_{\mathfrak{q}\mid\infty}\eta_\mathfrak{q}
$$
and
$$
\eta_\infty'=\Otimes\limits_{\mathfrak{q}\mid\infty}\eta_\mathfrak{q}',
$$
where the local components satisfy
$$
\eta_\mathfrak{q}\in
\left(
\bigwedge^{b_n}\liep_n^*\otimes\mathscr W_0(\pi_\mathfrak{q},\psi_\mathfrak{q})
\right)^{K_\mathfrak{q}}
$$
and
$$
\eta_\mathfrak{q}'\in
\left(
\bigwedge^{b_{n-1}}\liep_{n-1}^*\otimes\mathscr W_0(\sigma_\mathfrak{q},\overline{\psi}_\mathfrak{q})
\right)^{K_\mathfrak{q}'}.
$$
The index set $I$ decomposes into local sets $I_\mathfrak{q}$ and we have
$$
P(s)=\sum_{\#I=r_1b_n+r_2\tilde{b}_n}\varepsilon_I P_I(s),
$$
where
\begin{equation}
\Psi(w_{I,\infty}\otimes v_{I,\infty},s)=P_I(s)\cdot L(s,\pi_\infty\times\sigma_\infty).
\label{eq:Pdef}
\end{equation}
Equation \eqref{eq:Pdef} now reads
$$
P_I(s)=\prod_{\mathfrak{q}\mid\infty}\frac{\Psi(w_{I_\mathfrak{q},\mathfrak{q}},v_{I_\mathfrak{q},\mathfrak{q}},s)}{L(s,\pi_\mathfrak{q}\times\sigma_\mathfrak{q})}.
$$
The question wether
$$
P_\mathfrak{q}(\frac{1}{2}):=
\frac{\Psi(w_{I_\mathfrak{q},\mathfrak{q}},v_{I_\mathfrak{q},\mathfrak{q}},\frac{1}{2})}{L(\frac{1}{2},\pi_\mathfrak{q}\times\sigma_\mathfrak{q})}\neq 0
$$
is classical for $n=2$ and the case $n=3$ is treated in \cite[Theorem 3.8]{schmidt1993}. This proves
\begin{proposition}\label{prop:pnichtnull}
If $k$ is totally real, $n\in\{2,3\}$, and if $s=\frac{1}{2}$ is critical for $L(s,\pi\times\sigma)$, then $P(\frac{1}{2})\neq0$.
\end{proposition}

\section{The $\mathfrak{p}$-adic symmetric cube of elliptic curves}

Fix a totally real number field $k$ and assume that $E_1$ and $E_2$ are modular elliptic curves over $k$ (without CM). This is the case for example if $k/\QQ$ is solvable and $E_1$ and $E_2$ are defined over $\QQ$ \cite{taylorwiles1995,wiles1995,bcdt2001,book_langlands1980,book_arthurclozel1989}. We assume that $E_1$ and $E_2$ have good ordinary reduction at a finite place $\mathfrak{p}$ (assumed to be principal for simplicity). Then ${\rm Sym}^2 E_1$ is modular as well \cite{gelbartjacquet1978} and Theorems \ref{satz:distribution} and \ref{thm:algebraicity} yield the existence of a $\mathfrak{p}$-adic measure $\mu$ such that for any abelian Artin representation $\chi$ with $\mathfrak{p}$-power conductor $\mathfrak{f}$
$$
\int_\Gamma\chi d\mu=
\Omega\cdot\hat{\kappa}(\mathfrak{f})\cdot G(\chi)^{\frac{n(n-1)}{2}}\cdot L(2,{\rm Sym}^2 E_1\otimes E_2\otimes\chi),
$$
where $\Gamma$ denotes the Galois group of the maximal abelian extension of $k$, unramified outside $\mathfrak{p}$, and $\Omega\in\CC^\times$ by Proposition \ref{prop:pnichtnull}. $\mu$ naturally corresponds to an abelian $\mathfrak{p}$-adic $L$-function.

To show that $\Omega$ corresponds to Deligne's periods in the sense of \cite{yoshida1994}, we would like to appeal to \cite[Theorem 6.2]{garrettharris1993}. Unfortunately the parallel weight $2$ case is excluded in this theorem (cf. 3.4.8 and Remark 4.6.3 of loc. cit.; note that \cite[Theorem 4.3]{shimura1977} still holds in our case due to \cite{rohrlich1989}).

Now if $E_1=E_2=:E$ we have ${\rm Sym}^2E\otimes E={\rm Sym}^3E\oplus E(-1)$. Manin \cite{manin1976} showed the existence of the $\mathfrak{p}$-adic $L$-function $L_\mathfrak{p}(E,s)$ for $E$. If $E$ is already defined over $\QQ$ and $k/\QQ$ is abelian and if $\mathfrak{p}$ lies above a non-split rational prime $p$, it is easily seen that \cite{rohrlich1984} implies that $L_\mathfrak{p}(E,s)$ does not vanish identically. Therefore, with
$$
L_\mathfrak{p}({\rm Sym}^3 E,s):=\frac{L_\mathfrak{p}({\rm Sym}^2E\otimes E,s)}{L_\mathfrak{p}(E,s)},
$$
we get the $\mathfrak{p}$-adic $L$-function for ${\rm Sym}^3 E$ in this case.

Note that by \cite{kimshahidi2002a} ${\rm Sym}^3 E$ is modular as well so that the results of \cite{ashginzburg1994} are applicable. On the other hand Ash and Ginzburg must suppose the existence of a so-called $H$-model to construct the $\mathfrak{p}$-adic $L$-function for ${\rm Sym}^3 E$. The existence of the latter is equivalent to the condition that our $\mathfrak{p}$-adic $L$-function $L_\mathfrak{p}({\rm Sym}^3E,s)$ does not vanish identically (section 5 of loc. cit.). This remains an open problem, even in the case $k=\QQ$.

The pool of explicit examples for our theory is limited by the yet small number of proved instances of Langland's functoriality conjectures. We are optimistic that this situation will change in the future.

\bibliographystyle{spmpsci}

\end{document}